\documentclass[reqno, 11pt]{amsart}

\makeatletter
\@namedef{subjclassname@2020}{%
  \textup{2020} Mathematics Subject Classification}
\makeatother

\usepackage{amsfonts, amsthm, amsmath, amssymb}
\usepackage{hyperref}
\hypersetup{colorlinks=false}

\usepackage{bbm}

\usepackage[margin=1in]{geometry}

\usepackage{helvet}

\RequirePackage{mathrsfs} \let\mathcal\mathscr
   
\usepackage{hyperref} 
\usepackage{enumerate}

\usepackage{ dsfont }

\numberwithin{equation}{section}

\newtheorem{theorem}{Theorem}[section] 
\newtheorem{lemma}{Lemma}[section]

\theoremstyle{definition}
\newtheorem{example}{Example}[section]
 \newtheorem*{acknowledgements}{Acknowledgements}
\newtheorem*{remark}{Remark}
\newtheorem{definition}{Definition}[section]

\renewcommand{\d}{\mathrm{d}}
\renewcommand{\phi}{\varphi}

\renewcommand{\leq}{\leqslant}

\renewcommand{\geq}{\geqslant}

\renewcommand{\c}{\mathbf{c}}
\newcommand{\f}{\mathbf{f}}

\renewcommand{\b}{\mathbf{b}}

\renewcommand{\k}{\mathbf{k}}

\renewcommand{\r}{\mathbf{r}}

\usepackage{xcolor}

\newcommand{\e}{\psi}

\newcommand{\md}[1]{  \left(\textnormal{mod}\ #1\right)}

\newcommand{\Q}{\mathbb{Q}}

\newcommand{\N}{\mathbb{N}}
\newcommand{\R}{\mathbb{R}}

\newcommand{\Z}{\mathbb{Z}}
 \newcommand{\m}{\mathfrak{m}}
\renewcommand{\r}{\right}
\renewcommand{\b}{\mathbf}
\renewcommand{\c}{\mathcal}
\renewcommand{\epsilon}{\varepsilon}

\renewcommand{\leq}{\leqslant}
\renewcommand{\geq}{\geqslant}
\renewcommand{\#}{\sharp}

\DeclareMathOperator*{\Osum}{\sum{}^*}

\newcommand{\beq}[2]
{
\begin{equation}
\label{#1}
{#2}
\end{equation}
}

\title 
[Averages of arithmetic functions over  polynomials in many variables]
{Averages of arithmetic functions over  polynomials in many variables}

\author{Kevin Destagnol}  
\address{Laboratoire de mathématiques d'Orsay \\ Université Paris Saclay
\\   Orsay  \\  France}  
\email{kevin.destagnol@universite-paris-saclay.fr}

\author{Efthymios Sofos}  
\address{Mathematics Department    \\ Glasgow University
\\   Glasgow   \\ G12 8QQ  \\  UK}  
\email{efthymios.sofos@glasgow.ac.uk}


\begin{document}

\vspace{-2.5cm}

\begin{abstract}
We estimate the average of  any 
  arithmetic function $k$
  over the values of any smooth polynomial in many variables
  provided only that $k$ has a distribution  
 in  arithmetic  progressions
of fixed modulus.
We give several applications of this result including
 the analytic Hasse principle for an intersection
of two cubics in at least 21 variables
and 
asymptotics for 
the number of integer solutions of a non-algebraic variety.
\end{abstract}  

\maketitle
 
\vspace{-1cm}

\setcounter{tocdepth}{1}

\tableofcontents

\section{Introduction} \label{s:intro}  
Numerous challenges within the realm of Diophantine geometry can be rephrased in terms of computing averages of arithmetic functions, 
denoted as $k:\mathbb{N} \to \mathbb{C}$, over the values attained by integer polynomials. Examples include, among others, 
 the Hasse principle for conic bundles~\cite{MR3194818}, Manin's conjecture for Ch\^atelet surfaces~\cite{MR2874644}, \cite{RegisTim}, \cite{RegisGerald} and \cite{K} or prime and squarefree values of polynomials (see~\cite{HB}, \cite{FIw} or \cite{MR3996328} for example).

 To estimate the average of $k$ even over linear univariate polynomials 
one must plainly assume that $k$ has an average over every fixed  arithmetic progression $a \md q$. We may thus assume that 
$$ \sum_{\substack{ m\in \N \cap [1,x]\\ m\equiv a \md q}} k(m) \approx \rho(a,q) \int_1^x \omega(t) \mathrm dt
$$ for some $\rho(a,q) \in \mathbb C$ and a function $\omega$ of class $\mathcal{C}^1$. We may then use the Cram\'er--Granville model
to predict the average of $k$ over the values of any polynomial $f$
as done, for example in~\cite[Appendix]{MR3996328} 
when  $k$ is the indicator function of the primes. The statement 
analogous to~\cite[Eq. (A.1)]{MR3996328}  
is that the  conditional  expectation of
$k(m)$ given that $m\in [1,x]$ lies in the progression $a\md q $
is  $\approx q \rho(a,q) \int_1^x \omega(t) \mathrm dt$.
Letting $W$ be an integer that is divisible by all primes $p$
below some $z\to \infty$
the arguments in~\cite[pg. 32]{MR3996328}  
lead  to the 
heuristic  
\begin{align*} 
\sum_{ \substack{\b t \in (\Z\cap [-P,P])^n  \\ f(\b t) > 0}} 
k(f(\b t))  
&\approx  
\sum_{\b a \in (\Z/W\Z)^n}  
\sum_{\substack{  \\  \b t \in (\Z\cap [-P,P])^n\\ 
\b t \equiv \b a \md W}} W  \rho(f(\b t),W) \omega(f(\b t)) 
 \\&\approx  \int_{ [-P,P]^n} \omega( f(\b y) ) \mathrm d \b y \ \ \ 
 \sum_{\b a \in (\Z/W\Z)^n}  \frac{\rho(f(\b a),W)}{W^{n-1}}
.\end{align*} 
We will prove  this heuristic  for all polynomials $f$ in  sufficiently many variables.
The cases with a small number of variables are   harder;
in this case 
if one is allowed to assume that 
the   function is non-negative and has a weak multiplicative property then 
work of Nair--Tenenbaum~\cite{nairtenenbaum} 
or \cite{GeraldRegis} provide sharp upper bounds.

\subsection{The general result}Let  $\omega:[1,\infty)\to \mathbb C$ be  $\c C^1$ and for any  $q\in \N$  and $a\in \Z/q\Z$ let  $\rho( a,q)$ be a   number in $ \mathbb C$. Assume we are given an arithmetic function  $k:\N\to \mathbb C$  and define  \beq{eq:basicproperty}{E_{k,\omega,\rho}(x,q):= \sup_{y \in \R\cap  [1,  x]}
\max_{  a\in \Z/q\Z    }\left| \sum_{\substack{1\leq m \leq y   \\   m \equiv   a \md q  }} 
k( m) -  \rho(a,q)  \int_1 ^y  \omega( t) \mathrm d  t  \right| .} We shall henceforth denote $E_{k,\omega,\rho}(x,q)$ 
by $E(x,q)$ to simplify   notation. The function $k$ is \textit{equidistributed} in progressions modulo $q$ exactly when 
$E(x,q)$  has  smaller order of magnitude than $\int_1^x\omega$.
\begin{example}\label{ex:vonmnggg}Fix any $A>0$. When   $k$ is the indicator function of the primes
we choose $\omega(m)=1/\log (m+1)$ and $\rho(a,q)=\mathds 1(\gcd(a,q)=1)/\phi(q)$. The Siegel--Walfisz theorem~\cite[Eq.(5.77)]{iwan}  
is equivalent to   $$E(x,q)=O\left(\frac{ \int_1^x \omega(t) \mathrm dt}{(\log x)^A}\r)$$ with an implied constant independent of $q$.\end{example}

 \begin{definition}\label{def:box}Let  $f \in \Z[x_1,\ldots, x_n]$ be an integer polynomial.
Let  $\c B\subset  [-1,1]^n$ be of  the form $\c B =\prod_{j=1}^n [a_j,a'_j]$ for some $a_j$ with  $\max_j |a'_j-a_j|\leq 1 $ and define 
   \beq{def:minmax}{ b=2\max  \{|f(\b t)|: \b t \in \c B\}.}
\end{definition}

\begin{definition}\label{def:wzez} For any $z\geq 2$, assume we have a   function  
$m_p(z):\{\textrm{primes}\}\to \Z_{\geqslant 1}$ and let  $$W_z:=\prod_{p\leq z} p^{m_p(z)} .$$ 
We shall assume that  $ m_p(z)$ is suitably large so that $\widetilde{\epsilon}(z)\to 0 $ as $z\to \infty$, where
  $$\widetilde{\epsilon}(z):=\sum_{p\leq z} \frac{1}{p^{1+m_p(z)}} .$$
 \end{definition} The number   $\widetilde{\epsilon}(z)$ is an upper bound for the probability that a random integer is divisible by a
high power of at least one small prime. We abbreviate $m_p(z)$ by $m_p$.

Our main technical tool  expresses the error term  as a function of $E(x,W_z)$.

\begin{theorem}\label{lem:vachms} Assume that the positive integers  $d,n \in \N$  satisfy 
$n>2^d(d-1)$ and let $f \in \Z[x_1,\ldots, x_n]$  be a smooth form of degree $d$.
Let $s  \in \{-1,1\}$ and let  $\c B\subset \R^n$ be as in Definition~\ref{def:box}.
For  any  $z\geq 2 $ and $m_p$ as in Definition~\ref{def:wzez}, let $W_z$ be as in Definition~\ref{def:wzez} and assume that $\widetilde{\epsilon}(z)\to 0 $. 
Then there exist $c,\delta>0$ that only depend on $f$ such that   for any function $k:\N\to \mathbb C$ 
and all     large enough $P$ we have   \begin{align*}
\frac{1}{P^{n}} 
 \sum_{\substack{ \b t \in \Z^n\cap P\c B \\  s f(\b t ) > 0 }}  k(s f(\b t) ) 
&=
\Bigg(   \int\limits_{\substack{ \b t \in \c B \\ s f(\b t ) > P^{-d} } }  \omega( P^d s  f(\b t ))  \mathrm d \b t \Bigg)
 \sum_{\b t \in (\Z/W_z\Z )^{n} } \frac{ \rho( s f(\b t  ) ,W_z) }{ W_z^{n-1} }
\\
&+O\left( \frac{\|k\|_1 }{P^{d}  }   (P^{-\delta}+\widetilde{\epsilon}(z)+z^{-c}) +\frac{E(bP^d, W_z ) W_z  }{P^{d}   } \r),
\end{align*} where the implied constant is independent of $k$. Here
  $b$ is defined in~\eqref{def:minmax} and
$$  \|k\|_1:= \sum_{ \nu \in \N \cap [1 , b P^d]} |k( \nu)| 
.$$ \end{theorem} If the densities $\rho(a,W_z)$ have a multiplicative structure then the sum over $\b t$ is essentially a truncated Euler product over the 
primes $p\leq z $. 
The main idea of the proof is to not go through the usual route of major and minor arcs of the circle method but instead   
approximate the values of $f(\b t )$ by a Cram\'er--Granville  model. Let   $\mathrm e(z)=\mathrm e^{2\pi i z}$ for $z\in \mathbb C$. 
We shall show  in Lemmas~\ref{lesds8ghsd6d}-\ref{lecorelid6d}-\ref{lecocoreliada}  that  \beq{darksun}{  \int\limits_{\substack{ \b t \in \c B \\ s f(\b t ) > P^{-d} } }  \omega( P^d s  f(\b t ))  \mathrm d \b t =\int_{ \gamma   \in \R } \left( \int_{\c B} \mathrm e(\gamma f(\b t ) \mathrm d \b t)\right)
\int_{   [P^{-d}   , b  ]  } \omega( P^d  \mu)  \mathrm e(-   s \gamma    \mu     ) \mathrm d   \mu \mathrm d \gamma.}
This expression     is useful in   applications as it has $\omega$ and $ f $ separated. 

 \begin{remark}
The number of variables  can be reduced by half
 if one has better error terms regarding 
the distribution of $k$
on arithmetic progressions. This was done
in~\cite{MR3996328} for the indicator function of 
primes and square-free integers.
 \end{remark}

 \begin{remark}A convenient feature of Theorem~\ref{lem:vachms}
 is that it allows the user to make    choices for $z$ and the 
 exponents $m_p(z)$; this is useful in situations 
where   equidistribution for $k$ is easier    for moduli of specific factorisation.
For example, in certain Diophantine applications 
it is much easier to prove equidistribution for square-full moduli. Furthermore, $z$
is allowed to go to infinity arbitrarily  slow; this means that one only needs to prove 
equidistribution modulo very small moduli.
 \end{remark}

\subsection{A more accessible version} Theorem~\ref{lem:vachms} makes no assumptions on the arithmetic function~$k$. 
We give a version that is easier to use if $k$ has certain equidistribution properties:
\begin{theorem}\label{pergolesisalvereginainaminor} Assume that the positive integers  $d,n \in \N$  satisfy 
$n>2^d(d-1)$ and let $f \in \Z[x_1,\ldots, x_n]$  be a smooth form of degree $d$.
Let $s  \in \{-1,1\}$, let  $\c B\subset \R^n$ be as in Definition~\ref{def:box} and let $b$ be given by~\eqref{def:minmax}.
  Assume that  $k:\N\to \mathbb C$ is any function for which  
\begin{itemize}
\item $\displaystyle\sup_{q\in \N}\sum_{a=1}^q |\rho(a,q)|<\infty$ \ or \  $\rho(a,q)\in \mathbb R\cap[0,\infty)$ 
for all $q\in \N, a\in \Z/q\Z$;
\vspace{0.2cm}
\item $\displaystyle\int_1^x \omega(t) \mathrm dt$ is non-zero for all large enough $x$;
\vspace{0.2cm}
\item For each fixed $q\in \N$ we have $\displaystyle \lim_{x\to\infty}\frac{E(x,q)}{\displaystyle\int_1^x \omega(t) \mathrm dt}=0$;
\vspace{0.2cm}
\item We have $\displaystyle \liminf_{P\to\infty} 
P^d\Bigg|\int\limits_{\substack{ \b t \in \c B: s f(\b t ) > P^{-d} } }  \omega( P^d s  f(\b t ))  \mathrm d \b t \Bigg|\cdot 
\Bigg| \int_1^{bP^d} \omega(t)\mathrm dt\Bigg|^{-1} \neq 0$;
\vspace{0.2cm}
\item We have $\displaystyle  \limsup_{x\to\infty}\frac{\displaystyle\sum_{m\leq x}|k(m)|}{\displaystyle\left|\int_1^x \omega(t) \mathrm dt\right|}\neq \infty$.
\end{itemize} Define for any $z>1$ the integer    $T_z=\prod_{p\leq z}p^{t_p(z)}$, where $t_p(z)$ are arbitrary positive integers 
such that  for each fixed prime $p$ one has 
$\lim_{z\to\infty}t_p(z)=\infty$.
Then the   limit 
\beq{eq:george_cuoffee}{
\sigma(f):=\lim_{z\to\infty} T_z^{-n+1}  \sum_{\b t \in (\Z/T_z\Z )^{n} } \rho( s f(\b t  ) ,T_z)  }
exists and is independent of the choice of $t_p(z)$.
Furthermore, 
$$\lim_{P\to\infty}\frac{1}{P^{n}  
\displaystyle\int\limits_{\substack{ \b t \in \c B: s f(\b t ) > P^{-d}  } }  \omega( P^d s  f(\b t ))  \mathrm d \b t  } 
 \sum_{\substack{ \b t \in \Z^n\cap P\c B \\  s f(\b t ) > 0 }}  k(s f(\b t) ) 
=\displaystyle\sigma(f).$$  \end{theorem}

A useful feature of Theorem~\ref{pergolesisalvereginainaminor} is that we only assume 
 equidistribution in arithmetic progressions
of \textit{fixed} modulus and without an explicit error term.

\subsection{Analytic Hasse principle for intersections}
We say that a family of 
systems of polynomial equations with integer coefficients
satisfies 
the analytic Hasse principle  if the number of integer solutions 
in an expanding box converges to a constant, strictly positive, multiple 
of the analogous real   and   $p$-adic densities.
We will prove this       for a specific intersection of two degree $d$
polynomials. Our intersection is given by 
\beq{eq:intersection}{a^2+b^2=f(x_1,\ldots,x_n) , c^2+d^2 = 
1+f(x_1,\ldots, x_n ),}
where $f$ is smooth homogeneous of degree $d$ in $n>2^d(d-1)$ variables.
There is recent analytic  work in this area by
Rydin-Myerson~\cite{simon} who studied intersections of $R$ forms.
For intersections of two cubic forms this was recently
improved by Northey--Vishe~\cite{compo} by relaxing the assumption on the number 
of variables so as to work provided that there are at least 
 $39$. 
 For systems of two diagonal cubic forms
 work of 
 Br\"udern and   Wooley~\cite{bruwool} 
 proves asymptotics when one has at least $13$ variables.
The three simplest cases 
of our next result are 
 \begin{itemize}
 \item  two   quadratic equations in 
 at least $9$ variables,
 \item  two   cubic equations in 
 at least $21$ variables,
  \item  two    quartic equations in 
 at least $53$ variables.
 \end{itemize}  For 
 any box $\c B\subset [-1,1]^n$ and any 
 $P\geq 1 $ let 
$$N_f(P):=
\sum_{\substack{ \b x \in \Z^n\cap P \c B \\ f(\b x ) > 0 }} r(f(\b x )) r(1+f(\b x ) ) ,$$ where $r$ is the number of representations as the sum of two integer squares.
\begin{theorem}
\label{thm:analytic_HP} Assume that $f\in \Z[x_1,\ldots,x_n]$ is a smooth homogeneous degree $d$ polynomial in  $n>2^d(d-1)$ variables. For any box $\c B\subset [-1,1]^n$
inside of which $f$ only takes non-negative values we have 
$$\lim_{P\to\infty} \frac{ N_f(P)}{P^n}= \pi^2\mathrm{vol}\left(\{ \b t \in \c B \, : \, f(\b t)> 0\}\right)
\prod_{p} \sigma_p(f),$$ where 
$ \sigma_p(f)$ is given by $$\lim_{k\to\infty}p^{-k(n+2)}
 \#\left\{(\b t,x_1,x_2,y_1,y_2) \in (\Z/p^k\Z )^{n+4}:
x_1^2+x_2^2=f(\b t), y_1^2+y_2^2=1+f(\b t )\right\}
.$$
In particular, the analytic  Hasse Principle holds for~\eqref{eq:intersection} as soon as $f$ assumes at least one  strictly positive value.
\end{theorem}
The  proof relies on proving 
a shifted convolution result (Theorem~\ref{thm:thankyouthankyouthankyou})
and then feeding it into
Theorem~\ref{lem:vachms}.
The proof of Theorems~\ref{thm:thankyouthankyouthankyou} and~\ref{thm:analytic_HP}
are respectively in \S \ref{ss:convoltn}, \ref{ss:finishingproof} and \S\ref{ss:proof_thrm_convol}.

\subsection{Subsets of integers}
During the $2024$ conference on rational points in \textit{BIRS Chennai}
Pieropan \cite{report} asked the following (see \cite{report}): assume we 
are given a set $\c A\subset \Z$ whose density we can estimate 
as $$ \#\{a\in \c A: |a| \leq T\}=f(T)+O(g(T) ),$$ where $f(T)$ has the form $c_0 T^c (\log T)^e$ for some real constants $c_0,c$ and $e$ and where
$$\displaystyle\lim_{T\to\infty} g(T)/(T^c(\log T)^e)=0.$$ For an arbitrary polynomial $F\in \Z[x_1,\ldots, x_n]$ define 
$$ N(T;\c A,F)=\#\left\{\b x \in \Z^n\cap [-T,T]^n: F(\b x ) \in \c A\right\}.$$ The question is under what conditions on $\c A,f,g,n$ and $\deg(F)$ can one prove asymptotics for $ N(T;\c A,F)$ as $T\to\infty$?

Our result states that the condition $n>2^{\deg(F)} (\deg(F)-1)$ is sufficient, as long as one has the additional assumption 
that we can count asymptotically the  elements of $\c A$ on any fixed arithmetic progression $r\md q$. 
To see why this is necessary, just consider the polynomial $F(x)=r+qx$.
For convenience of notation we will 
assume that $\c A$ consists of strictly positive integers; in applications
this does not cause any problems as one can split in cases according to the sign.

\begin{theorem}\label{pieropanwn} Assume that the positive integers  $d,n \in \N$  satisfy 
$n>2^d(d-1)$ and let $F \in \Z[x_1,\ldots, x_n]$  be a smooth form of degree $d$.
Let    $\c B\subset \R^n$ be as in Definition~\ref{def:box} and let $b$ be given by~\eqref{def:minmax}.
 Assume that  $\c A\subset \Z_{\geqslant 0}$ is a non-empty set for which there exists $\omega:[1,\infty)\to (0,\infty)$ in  $\c C^1$
and for all $q\in \N$ and $r\in \Z/q\Z$ a real number $\rho(q,r)\geq 0$ such that we have 
$$\lim_{x\to\infty}
\sup_{y \in \R\cap  [1,  x]}
\max_{  a\in \Z/q\Z    }\left| \frac{ \#\{a\in \c A: a \leq y, a \equiv   r \md q   \} }{\displaystyle\int_1 ^y  \omega( t) \mathrm d  t }
 -  \rho(r,q)   \right| 
=0.$$ 
Assume, in addition, that $$ \liminf_{P\to\infty} 
P^d\Bigg|\int\limits_{\substack{ \b t \in \c B: | F(\b t )| > P^{-d} } }  \omega( P^d    F(\b t ))  \mathrm d \b t \Bigg|\cdot 
\Bigg| \int_1^{bP^d} \omega(t)\mathrm dt\Bigg|^{-1} \neq 0.$$ 
Define for any $z>1$ the integer    $T_z=\prod_{p\leq z}p^{t_p(z)}$, where $t_p(z)$ are arbitrary positive integers 
such that  for each fixed prime $p$ one has  $\lim_{z\to\infty}t_p(z)=\infty$.
Then the   limit $$\sigma(F):=\lim_{z\to\infty} T_z^{-n+1}  \sum_{\b t \in (\Z/T_z\Z )^{n} } \rho(   F(\b t  ) ,T_z)  $$ exists and is independent of the choice of $t_p(z)$.
Furthermore,  $$\lim_{P\to\infty}\frac{\#\{\b x \in \Z^n\cap P \c B: F(\b x ) \in \c A \}}{P^{n}  
\displaystyle\int\limits_{\substack{ \b t \in \c B: F(\b t )> P^{-d}  } }  \omega( P^d    F(\b t ))  \mathrm d \b t  } 
 =\displaystyle\sigma(F).$$ 
\end{theorem}  
The proof is given in \S\ref{ss:pieropnapl} and is a direct 
application of Theorem~\ref{pergolesisalvereginainaminor}.

\subsection{Points on non-algebraic varieties}
Exponential Diophantine equations 
form a considerable 
area of research in number theory
with many results   directed at 
proving finiteness of solutions,
see for example the work of 
Bugeaud--Mignotte--Siksek~\cite{11021} 
and the book of 
Shorey and Tijdeman~\cite{060610011}. Here 
we change perspective and     count
 asymptotically the number of 
 integer solutions of
$$
2^{x_0} =f(x_1,\ldots, x_n),
$$ where $f$ is an integer polynomial in many variables.   
\begin{theorem}
\label{thm:nonalgebraic} For every smooth homogeneous $f\in \Z[x_1,\ldots, x_n]$ of degree $d$ with $n>2^d(d-1)$
and  for every box $\c B\subset \R^n$ as in Definition~\ref{def:box}   inside which $f$ assumes at least one  strictly positive value,
we have $$\lim_{P\to\infty}  \frac{\#\left\{ \b x \in \Z^n\cap P \c B, x_0 \in \N \cap\left[1,\log (bP^{d})/\log 2\right]: f(\b x )=2^{x_0} \right\}}{P^{n-d} \log P }   
=\frac{d\sigma_\infty(f)}{\log 2}  \sigma_2(f) \prod_{\substack{ p>2  }} \sigma_p'(f)
,$$ where $\sigma_\infty(f), \sigma_2(f)$ are the usual
Hardy--Littlewood real and $2$-adic densities associated to $f=0$, 
 $b$ is as in~\eqref{def:box}
 and for an odd prime $p$ we have 
 $$ \sigma_p'(f):= 
\lim_{m\to \infty}
 \frac{\#\left\{\b y \in (\Z/p^{m}\Z )^{n}, h   \in \Z/\phi(p^{m})\Z  :  f(\b y )   =  2^ h \right\}}{(p-1) p^{mn  } }
. $$   
\end{theorem} 
The proof is given in \S\ref{ss:proofnonalg}
by applying Theorem~\ref{pergolesisalvereginainaminor}.

\subsection{Chowla's conjecture}
Let $\mu$ denote the M\"obius function.
Chowla's conjecture states that for any $f\in \Z[x]$ for which $f\neq c_1g^2$ for  $c_1\in \Q$ and $g\in \Z[x]$ one has 
$$ \frac{1}{P} \sum_{\substack{ m\in \N\cap [1,P]  \\ f( m  ) \neq 0 }} \mu(|f(\b m) |) \to 0$$ as $P\to \infty$. It has been the focus of 
intensive investigation recently, see the work of Matomäki--Radziwi\l\l--Tao~\cite{marata}  and Tao--Teräväinen~\cite{tate}, for example.
We prove the analogue of the conjecture for all polynomials in a sufficiently large number of variables.
\begin{theorem}\label{thm:chowla}Chowla's conjecture holds for all smooth homogeneous $f\in \Z[x_1,\ldots, x_n]$ of degree~$d$ when $n>2^d(d-1)$.
In particular,   for every box $\c B\subset \R^n$ as in Definition~\ref{def:box}  
we have $$\lim_{P\to\infty} \frac{1}{P^n} \sum_{\substack{ \b x \in \Z^n\cap P \c B \\ f(\b x ) \neq 0 }} \mu(|f(\b x ) |) =0.$$\end{theorem}
The proof is given in \S\ref{ss:claa}.

\subsection{Sums of Fourier coefficients of modular forms over values of multivariable polynomials}
Asymptotics for averages of 
coefficients of modular forms over values of irreducible quadratic  polynomials
in one variable have been studied by
Blomer~\cite{MR2435749}, Templier~\cite{MR2783929}
and   Templier--Tsimerman~\cite{MR3096570}.
Bounds in the case of higher degree polynomials were later given by
Chiriac and Yang~\cite{MR4480204}.

Let $f$ be a holomorphic cusp form of weight $k$ for $\textrm{SL}_2(\Z)$ with Fourier expansion 
$$f(z)=\sum_{m=1}^\infty \lambda_f(m) m^{(k-1)/2} \mathrm{e}(m z)$$ for $\Im(z)>0$.
Let  $f$ be a normalized Hecke eigenform so that $\lambda_f (1) = 1$.
We have the bound $\lambda_f(n)\ll \tau(n)$ 
by the work the Ramanujan–Petersson conjecture,
proved by Deligne~\cite{deligne} when  $k\geq 2 $
and by   Deligne--Serre\cite{delser} when $k=1$. 

\begin{theorem}\label{thm:chowla2}For all smooth homogeneous 
$F\in \Z[x_1,\ldots, x_n]$ of degree~$d$ with $n>2^d(d-1)$
and all  boxes $\c B\subset \R^n$ as in Definition~\ref{def:box}  
we have $$\lim_{P\to\infty} \frac{1}{P^n} 
\sum_{\substack{ \b x \in \Z^n\cap P \c B \\ 
F(\b x ) \neq 0 }}
\lambda_f(|F(\b x ) |) =0.$$\end{theorem}
The proof is given in \S\ref{ss:claa2}.

\subsection{Divisor function over values of polynomials}
Let us continue with estimating the average of the divisor function $\tau$ over polynomials. 
In the case of one variable polynomials,   asymptotics 
are known only for polynomials of degree $1$ and $2$ (see Hooley~\cite{hooley}), 
while, for binary forms asymptotics are known for degree $\leq 4 $ by work of Daniel~\cite{daniel}. 
For a prime $p$ and positive integers $m,\ell $
we define $ \tau_{p^m}(\ell ):= 1+\min\{v_p(\ell ),m\}$. This   is a 
local model for the divisor function $\tau$ since $\tau(\ell)
=\prod_{p\mid \ell  }\lim_{m\to\infty} \tau_{p^m}(\ell )$.
\begin{theorem}\label{thm:divisorddd}For every smooth homogeneous $f\in \Z[x_1,\ldots, x_n]$ of degree $d$ with $n>2^d(d-1)$
and  for every box $\c B\subset \R^n$ as in Definition~\ref{def:box}   
we have $$ \lim_{P\to\infty} \frac{1}{P^n\log P }  \sum_{\substack{ \b x \in \Z^n \cap P \c B \\ f(\b x ) \neq 0 }} \hspace{-0,2cm}
\tau(|f(\b x ) |) =d  \mathrm{vol}(\c B) \prod_{p}  \left(1-\frac{1}{p}\r)
\lim_{m\to\infty}\frac{1}{p^{n m}} \sum_{\b t \in (\Z/p^{m } )^{n}} \tau_{p^m}(f(\b t ))
.$$\end{theorem}
The proof is given in \S\ref{s:dvsrcrp} as an application of 
an equidistribution result of  
Pongsriiam--Vaughan~\cite{MR3352438}
and Theorem~\ref{pergolesisalvereginainaminor}.

\subsection{Rational points on a hyperelliptic hypersurface}
For an integer polynomial $f$ of degree at least $ 5$ the hyperelliptic curve $y^2=f(x)$ has finitely many rational points by Faltings' theorem \cite{MR0718935}.
On the contrary when $f$ has many variables one expects infinitely many rational points; here we shall give an asymptotic estimate.
 
 \begin{theorem}\label{thm:hyperelipt}Fix an integer $k\geq 2$. For every smooth homogeneous $f\in \Z[x_1,\ldots, x_n]$ of degree $d$ with $n>2^d(d-1)$
and  for every box $\c B\subset \R^n$ as in Definition~\ref{def:box}   inside which $f$ assumes at least one  strictly positive value,
we have \begin{align*}
\lim_{P\to\infty}  &\frac{\#\left\{ \b x \in \Z^n \cap P \c B: f(\b x ) \textrm{ is a }k\textrm{-th power of a positive integer}\right\}}{P^{n-d(1-1/k) } }   
\\ &= 
\prod_{p} \lim_{m\to  \infty }  \frac{\#\{y,\b t \in (\Z/p^m\Z)^{n+1}: y^k= f(\b t )  \}}{p^{m n }}
\lim_{\delta\to 0_+}\int\limits_{\substack{ \b t \in \c B \\   f(\b t ) > \delta } }  \frac{\mathrm d \b t }{k    f(\b t )^{1-1/k}} .\end{align*}
\end{theorem}
The proof is given in \S\ref{s:pwerdd} as an 
application of Theorem~\ref{pergolesisalvereginainaminor}.

\subsection{$m$-full numbers represented by polynomials} Let $m\in \N$, $m\geq 2$. An integer is called $m$-full if it is divided by the $m$-th power 
of every of its prime divisors. Values of polynomials that are $m$-full have been studied by various authors; for example, Pasten~\cite{MR3057059}
proved that an integer polynomial should take very few  $m$-full values conditionally on Vojta's conjecture and have applications to the study of Campana points (see for example \cite{DamarisMarta}). The picture for multivariable polynomials 
is rather different; indeed here one expects infinitely many solutions. We provide asymptotics.

 \begin{theorem}\label{thm:mful}Fix an integer $m\geq 2$. For every smooth homogeneous $f\in \Z[x_1,\ldots, x_n]$ of degree~$d$ with $n>2^d(d-1)$
and  for every box $\c B\subset \R^n$ as in Definition~\ref{def:box}   inside which $f$ assumes at least one  strictly positive value,
we have $$\lim_{P\to\infty}  \frac{\#\left\{ \b x \in \Z^n \cap P \c B : f(\b x ) \textrm{ is }m\textrm{-full and positive}\right\}}{P^{n-d(1-1/m) } }   
=\sigma(f) \lim_{\delta\to 0_+}\int\limits_{\substack{ \b t \in \c B \\   f(\b t ) > \delta } }  \frac{\mathrm d \b t }{m    f(\b t )^{1-1/m}}
,$$ where $\sigma(f)$ is given by $$ \sum_{\substack{ k_2,\ldots, k_m \in \N   } } \frac{ \mu(k_2\cdots k_m)^2}{
 k_2^{1+1/m} k_3^{1+2/m} \cdots k_{m}^{ 2-1/m }  } 
 \prod_{p} 
 \lim_{m\to \infty }  \frac{\#\{y,\b t \in (\Z/p^m\Z)^{n+1}: 
y^m k_2^{m+1}\cdots k_{m}^{2m-1} = f(\b t )  \}}{p^{m n }} .$$\end{theorem}
The proof is given in \S \ref{s:fulcrp} as an application of 
Theorem~\ref{pergolesisalvereginainaminor}.

\begin{acknowledgements}We would like to thank   Ping Xi for asking questions during a  seminar in 
Xi’an Jiaotong University
that started activities that   led to Theorem \ref{lem:vachms}. Parts of this 
investigation too place 
when the second named author visited the university Paris-Saclay during April $2022$,
 the generous hospitality and support of which  is greatly appreciated.
\end{acknowledgements}

\section{Proof of 
Theorems~\ref{lem:vachms} and \ref{pergolesisalvereginainaminor}} 
The Hardy--Littlewood singular series for the problem of representing an integer $\nu$ by the polynomial   $f$ is given by  
$$\mathfrak S(\nu):=\prod_{\substack{ p=2 \\ p \textrm{ prime}}}^\infty \lim_{m\to \infty} 
\frac{\#\{\b t \in (\Z/p^{m}\Z )^{n}:  f(\b t ) \equiv \nu \md{p^m} \}}{p^{m(n-1) }}.$$ We start by approximating it by its multiplicative model
$$\mathfrak S^\flat(\nu):= \frac{\#\{\b t \in (\Z/W_z\Z )^{n}:  f(\b t ) \equiv \nu \md{p^m}  \}}{W_z^{n-1 }} .$$ 
\begin{lemma}\label{lem:cramer}Let $f$ be as  in     
Theorem~\ref{lem:vachms} and assume that $ \widetilde{\epsilon}(z) \to 0$. 
There exists $c=c(f)>0$ such that for all $\nu \in \Z$ we have 
$$\mathfrak S(\nu)-\mathfrak S^\flat(\nu)\ll  \widetilde{\epsilon}(z) +z^{-c},$$ where the implied constant is independent of $\nu$.
\end{lemma}\begin{proof}By  \cite[\S 7]{MR0150129} we have 
$$\lim_{m\to\infty}\frac{\#\{\b t \in (\Z/p^{m}\Z )^{n}:  f(\b t ) = \nu  \}}{p^{m(n-1) }}
= 1+  \sum_{ r\in \N  }  p^{-rn} \sum_{\substack{  a \in  \Z/p^r\Z \\ p\nmid  a  } }S_{  a,p^r}\mathrm e (- a \nu  /p^r) ,$$ where  
\beq{defbaH}{  S_{ a,q} :=\sum_{\b t \in (\Z/q\Z)^n}\mathrm e \left(  \frac{a}{q} f (\b t ) \r).}  The assumption $n >  2^d (d-1)$ shows that the constant
\beq{eq:deprez}{c:=   \frac{n 2^{-d}}{(d-1)}-1 } is strictly positive. 
Using   \cite[pg. 256, Eq. (19)]{MR0150129}  with $\epsilon =c$ 
we see that \beq{son234}{|S_{a,p^r}| \ll  p^{r \left( n-c-2 \r)}}
with an implied constant that depends at most on $f$. We infer that for all $m \geq 1 $ one has 
\beq{eq:kotozimi12}
{\left| \sum_{ r\geq m  }  p^{-rn} \sum_{\substack{  a \in \Z/p^r\Z\\ p\nmid a  } }S_{ a,p^r}\mathrm e (- a  \nu  /p^r) \r| \ll  
 \sum_{ r\geq m  }  p^{-r(1+c) }  \ll p^{-m(1+c) } ,}
where the implied constant depends at most on $ f $.
In the special case that $m=1$ this yields 
\beq{eq:unifmcaf}{ \left|\lim_{m\to\infty}\frac{\#\{\b t \in (\Z/p^{m}\Z )^{n}:  f(\b t ) \equiv \nu \md{p^m}  \}}{p^{m(n-1) }}
-1\r | \ll  p^{-(1+c)}} uniformly in $\nu$.
Hence, for any $z>1$ we have 
$$ \log \prod_{p>z} \lim_{m\to\infty}\frac{\#\{\b t \in (\Z/p^{m}\Z )^{n}:  f(\b t ) \equiv \nu \md{p^m}  \}}{p^{m(n-1) }} 
= \sum_{p>z} \log \left(1+O(p^{-1-c})  \r)  \ll \sum_{p>z} p^{-1-c} \ll z^{-c}$$
and therefore
$$
\prod_{p>z} \lim_{m\to\infty}\frac{\#\{\b t \in (\Z/p^{m}\Z )^{n}:  f(\b t ) \equiv \nu \md{p^m}  \}}{p^{m(n-1) }} =1+O\left(z^{-c}\right).
$$
We have $\mathfrak S(\nu)=O(1)$ independently of $\nu$ 
due to~\cite[Corollary, pg.256]{MR0150129}. Thus, $$ \mathfrak S(\nu)= 
\left( \prod_{p\leq z} \lim_{m\to\infty}\frac{\#\{\b t \in (\Z/p^{m}\Z )^{n}:  f(\b t ) \equiv \nu \md{p^m}  \}}{p^{m(n-1) }}\r) +O(z^{-c})$$
with an implied constant independent of $\nu$. By \eqref{eq:kotozimi12} we see that  for any   $m_p\geq 1$,  
$$  \lim_{m\to\infty}\frac{\#\{\b t \in (\Z/p^{m}\Z )^{n}:  f(\b t ) \equiv \nu \md{p^m}  \}}{p^{m(n-1) }}-
\sum_{ 0\leq r \leq m_p }  p^{-rn} \!\!\!\sum_{\substack{  a \in \Z/p^r\Z\\ p\nmid a  } }S_{a,p^r}\mathrm e (-  a   \nu  /p^r)
\ll   p^{-(1+c) (1+m_p)}.$$  This means that 
$$\prod_{p\leq z}  \lim_{m\to\infty}\frac{\#\{\b t \in (\Z/p^{m}\Z )^{n}:  f(\b t ) \equiv \nu \md{p^m}  \}}{p^{m(n-1) }}=V  
 \prod_{p\leq z} \left ( \sum_{ 0\leq r \leq m_p }  p^{-rn} \sum_{\substack{  a \in \Z/p^r\Z\\ p\nmid  a  } }
S_{a,p^r}\mathrm e (-a  \nu  /p^r)\r),$$ for some quantity $V$ that satisfies 
 $$V=\prod_{p\leq z} \left(1+O\left (p^{-(1+c) (1+m_p)}\r) \r)=
\exp\left(O\left(\sum_{p\leq z } p^{-(1+c) (1+m_p)} \r)\r) =1+O(\widetilde{\epsilon}(z)) $$ due to the assumption $\widetilde{\epsilon}(z)\to 0 $.
Recalling that  $ W_z=\prod_{p\leq z}p^{m_p},$ the last equation in  \cite[pg. 259]{MR0150129} gives
\beq{eq:touselater}{ \sum_{ 0\leq r \leq m_p }  p^{-rn} \sum_{\substack{  a \in \Z/p^r\Z\\ p\nmid  a  } }S_{ a,p^r}\mathrm e (- a \nu  /p^r)
=\frac{\#\{\b t \in (\Z/p^{m_p} )^{n}:  f(\b t ) \equiv \nu \md{p^m}  \}}{p^{(n-1) m_p}},}   which concludes the proof via the Chinese remainder theorem.
\end{proof}The Hardy--Littlewood singular integral for       representing     $\nu$ by       $f$ is   
$$ J(\nu):=\int_{\R} I(\c B;\gamma)\mathrm e(- \gamma \nu    ) \mathrm d \gamma,
 \ \ \ \textrm{ where }  \ \ \ I(\c B;\gamma):=\int_{\c B} \mathrm e \left (  \gamma f(\b t ) \r )\mathrm d  \b t   . $$

We are now in position to start averaging the arithmetic function over the polynomial values.
\begin{lemma}
\label{lem:letsstart}Let $f$ be as  in      Theorem~\ref{lem:vachms} and assume that $ \widetilde{\epsilon}(z) \to 0$. 
There exists $c=c(f)>0$ such that for all large $P$ one has 
\begin{align*} \sum_{\substack{ \b t \in \Z^n\cap P\c B \\  s f(\b t ) > 0 }}  k(sf(\b t) ) 
=&  \frac{P^{n-d} }{W_z^{n-1}} \sum_{\b t \in (\Z/W_z )^{n} }  \int_{\R  } I(\c B;\gamma)
 \left\{ \sum_{ \substack{ \nu\in \N\cap[1,b P^d] \\ \nu \equiv s f(\b t ) \md{W_z} }} k( \nu)  
  \mathrm e(- \gamma s \nu P^{-d}    ) \r\} \mathrm d \gamma
\\  +&O\left( P^{n-d} \left( P^{-\delta} +  \widetilde{\epsilon}(z)  + z^{-c}  
 \r) 
\|k\|_1 \r).
\end{align*}
\end{lemma}
\begin{proof}For    $P$ suitable large and all $\b t \in P\c B$ we have  $ |f(\b t)| \leq b P^d$, hence,
\beq{eq:doubleviolinconcert}{\sum_{\substack{ \b t \in \Z^n\cap P\c B \\  s f(\b t ) > 0 }}  k(sf(\b t) ) 
=  \sum_{ \nu\in \N\cap[1,b P^d]} k( \nu) \#\{\b t \in \Z^n \cap P\c B  :  f(\b t)= s  \nu \}.}
By~\cite[Lem.5.5]{MR0150129} there is $\delta=\delta(  f )>0$ such that  $$\#\{\b t \in \Z^n \cap P\c B  : f (\b t)= s   \nu \} =
P^{n-d} \mathfrak S(s \nu) J(s \nu P^{-d})+O(P^{n-d-\delta})
,$$ where the implied constant is independent of $s  $ and $\nu $. 
By Lemma~\ref{lem:cramer}
we can write the right-hand side as   $$ P^{n-d}  \mathfrak S^\flat(s\nu)  
\int_{\R} I(\c B;\gamma)\mathrm e(- \gamma s \nu P^{-d}    ) \mathrm d \gamma  
+O\left(P^{n-d} \left( P^{-\delta} +  \widetilde{\epsilon}(z)  +z^{-c} 
 \r) \r),$$ where we used that $J(s \nu P^{-d})$ is bounded only in terms of $f$.
Substituting into~\eqref{eq:doubleviolinconcert} we obtain the claimed error term. In the  ensuing main term 
we can interchange the order of summation and integration owing to the fact that 
$ \int_{\R} | I(\c B;\gamma)|\mathrm d \gamma  <\infty $ by   \cite[Lem. 5.2]{MR0150129}. This gives 
$$P^{n-d} \int_{\R } I(\c B;\gamma) \left\{
\sum_{ \nu\in \N\cap[1,b P^d]} k( \nu)  \mathfrak S^\flat(s\nu)  \mathrm e(- \gamma s \nu P^{-d}    ) \r\}
\mathrm d \gamma.$$ By definition of the function $ \mathfrak S^\flat(s\nu)  $, 
we obtain $$P^{n-d} W_z^{-n+1} \int_{\R } I(\c B;\gamma)
\sum_{\b t \in (\Z/W_z\Z )^{n} }  \left\{ \sum_{ \substack{ \nu\in \N\cap[1,b P^d] \\ \nu \equiv s f(\b t ) \md{W_z} }} k( \nu)  
  \mathrm e(- \gamma s \nu P^{-d}    ) \r\} \mathrm d \gamma,$$ as claimed.  \end{proof} The presence of $ I(\c B;\gamma)$ in the main term 
shows that only small $\gamma$ matter. Therefore, a simple approach to deal with those will suffice. 
\begin{lemma}\label{lem:partsum}For any $x\geq 1$, $a,q\in \mathbb Z$ with $q\neq 0$ and any $\mu\in  \R$ we have 
$$
\sum_{\substack{ 1\leq \nu \leq x \\ \nu \equiv a \md q  }} k(\nu ) \mathrm e (\mu \nu )=
\rho( a,q) \int_1^x \omega(t)\mathrm e (\mu t) \mathrm dt
+O\big((1+ |\mu | x) E(x,q) \big)
,$$   where the implied constant is absolute.
\end{lemma}\begin{proof}By partial summation we obtain 
$$ A(x)  \mathrm e (\mu x )-2 \pi i \mu \int_1^x A(t) \mathrm e (\mu t ) \mathrm dt,  \ \ \ \textrm{ where } \ \ \ 
 A(t):= \sum_{\substack{ 1\leq \nu \leq t \\ \nu \equiv a \md q  }} k(\nu ).$$ 
Alluding to~\eqref{eq:basicproperty}  we then obtain 
$$ 
\rho( a,q) 
\left\{ 
 \mathrm e (\mu x )  \int_1 ^x  \omega( t) \mathrm d  t   
-2 \pi i \mu  
\int_1^x 
\int_1 ^t  \omega( u) 
\mathrm e (\mu t ) \mathrm d  u \mathrm dt
\r\}
+O\left((1+ |\mu | x) E(x,q) \r)
$$ and infer that the difference inside the brackets equals      $\displaystyle
\int_1^x 
\omega(t)
\mathrm e (\mu t) \mathrm dt
.$  \end{proof}
Define $$ H:=\int_\R  I(\c B;\gamma) \bigg( \int_{P^{-d}}^{b} \omega(P^d y) \mathrm e (-\gamma s y )\mathrm d y\bigg) \mathrm d \gamma .$$
\begin{lemma}\label{lem:guitar}Keep the setting of Lemma~\ref{lem:letsstart}.
There exists $c=c(f)>0$ such that for all large $P$ one has 
\begin{align*} \sum_{\substack{ \b t \in \Z^n\cap P\c B \\  s f(\b t ) > 0 }}  & k(sf(\b t) ) 
- \frac{P^{n} H }{W_z^{n-1}}  \sum_{\b t \in (\Z/W_z )^{n} } \rho(sf(\b t),W_z) 
\\  \ll & P^{n-d}
W_zE(bP^d,W_z)+P^{n-d} \left( P^{-\delta} +  \widetilde{\epsilon}(z)  + z^{-c}  
 \r) \|k\|_1 .\end{align*}\end{lemma} \begin{proof}
 We inject Lemma~\ref{lem:partsum} into 
Lemma~\ref{lem:letsstart}. The contribution of the error term is 
$$ \ll 
\frac{P^{n-d} }{W_z^{n-1}} \sum_{\b t \in (\Z/W_z )^{n} }  \int_{\R  } |I(\c B;\gamma)|
\left(1+\left|\gamma s P^{-d}\right|P^d\right) E(bP^d,W_z)\mathrm d \gamma
.$$ By \cite[Lem. 5.2]{MR0150129} with $\epsilon$ being the positive      constant $c$  in~\eqref{eq:deprez}
we obtain   \beq{bchrand4}{| I(\c B;\gamma) | \ll \min\{1,|\gamma |^{-c-2} \},} hence, $\displaystyle\int_\R | I(\c B;\gamma) | (1+|\gamma|) \mathrm d \gamma<\infty$. This gives the estimate 
$$ \ll E(bP^d,W_z)
\frac{P^{n-d} }{W_z^{n-1}} \sum_{\b t \in (\Z/W_z )^{n} }  \int_{\R  } I(\c B;\gamma)
(1+|\gamma |) \mathrm d \gamma\ll E(bP^d,W_z)P^{n-d} W_z 
.$$

 \end{proof}
To deal with   $H$   we generalise the Fourier analysis approach in~\cite{MR3996328}. 
For $k\in \N$ let $\phi_k:\R\to \R$ be 
$ \phi_k(\gamma ) = \pi^{-1/2} k \exp(-k^2 \gamma ^2),$ whose Fourier transform is
$ \widehat{\phi}_k(\gamma ) =\exp(-\pi^2 k^{-2}  \gamma^2)$.
Define $$H_k=\int_{ \gamma   \in \R }  \widehat{\phi}_k(\gamma ) I(\c B;\gamma)
\int_{   P^{-d} }^b \omega( P^d  y)  
\mathrm e(-  \gamma   s  y  )
\mathrm d  y
\mathrm d \gamma.$$

\begin{lemma}\label{lem:nbnd}For fixed $P\geq 1$ 
we have  $ H_k\to H$ as $k\to +\infty$. \end{lemma}\begin{proof}By~\eqref{bchrand4} and  $|\widehat{\phi}_k(\gamma )|=O(1) $
one   sees that 
$$\lim_{k\to\infty}
\int_{|\gamma|> \log k } ( \widehat{\phi}_k(\gamma ) -1) I(\c B;\gamma)
\int_{   P^{-d} }^b \omega( P^d  y)  
\mathrm e(-  \gamma   s  y  )
\mathrm d  y
\mathrm d \gamma
=0.$$    In the region $ |\gamma|\leq \log k $ one has 
$$ \widehat{\phi}_k(\gamma ) =\exp(-\pi^2 k^{-2}  \gamma^2)=1+O(k^{-2}  \gamma^2 )
 =1+O\left(\frac{\log^2 k}{k^2}  \r),$$ with an implied constant that is independent of $\gamma$. Hence, $$
\int_{|\gamma|> \log k } ( \widehat{\phi}_k(\gamma ) -1) I(\c B;\gamma)
\int_{   P^{-d} }^b \omega( P^d  y)  
\mathrm e(-  \gamma   s  y  )
\mathrm d  y
\mathrm d \gamma
\ll\frac{\log^2 k}{k^2}  \int_{\R  } |  I(\c B;\gamma)|
 \mathrm d \gamma\ll\frac{\log^2 k}{k^2},$$ which concludes the proof. \end{proof}
For $k\in \mathbb N$ and $\b t \in \c B$ let  $$ F_k (\b t ) : = 
 \int\limits_{\substack{  y \in   [P^{-d}    , b  ] \\ 
 | f  (\b t ) -sy |
\leq 1/\sqrt k } } \omega( P^d y) 
 \phi_k (  f (\b t )-sy )
 \mathrm d   y.$$ 
\begin{lemma}\label{lesds8ghsd6d}We have   $$H=
  \int_{ \c B} \lim_{k\to \infty} F_k (\b t ) \mathrm d t  .$$ \end{lemma}\begin{proof}
By the definition of $I(\c B;\boldsymbol\gamma)$ we can write $$H_k= \int_{\b t \in \c B} 
\int_{  y \in    [P^{-d}   , b  ] } \omega( P^d y ) 
 \int_{ \gamma   \in \R  }  \widehat{\phi}_k(\gamma )  
   \mathrm e ( \gamma  \left\{ f(\b t )- s y\r\} )
\mathrm d \gamma
\mathrm d  y \mathrm d \b t   
,$$ where we swapped the order of integration owing to $\int_\R  | \widehat{\phi}_k(\boldsymbol\gamma )  | \mathrm d \gamma <\infty$.
By Fourier inversion   $$H_k= \int_{\b t \in \c B}  \int_{  y \in    [P^{-d}   , b  ] } \omega( P^d y ) 
 \phi_k   (  f(\b t )-sy ) \mathrm d  y \mathrm d \b t  .$$ For $| y | > 1/\sqrt k$ we see that  \beq{eq:ernm}{ \phi_k( y ) = 
\pi^{-1/2} k\exp(-k^2y^2)\ll k \exp(-k  ) .} 
Hence, by Lemma~\ref{lem:nbnd} we get  $$H=\lim_{k\to \infty}
  \int_{ \c B} F_k (\b t ) \mathrm d t  .$$To conclude the proof it will suffice to    use the dominated convergence theorem.
Its assumptions hold since 
$F_k(\b t )$ is bounded independently of $k$ and $\b t $. Indeed,  
$$\left | F_k(\b t ) \r| \ll \left (\max_{ y \in   [P^{-d}    , b  ] } |\omega( P^dy) |\r)
 \int_{\substack{  \R    } }  
 \phi_k   (  f (\b t )-sy )
 \mathrm d y $$ and   note that a  linear change of variables makes    the integral equal to 
  $ \displaystyle\int_\R  \phi_k(z) \mathrm d z=1$.  
\end{proof}

\begin{lemma}\label{lecorelid6d}We have   $$H= \int\limits_{\substack{ \b t \in \c B \\ s f(\b t) > P^{-d} }}
   \lim_{k\to \infty} F_k (\b t ) \mathrm d t  .$$ \end{lemma}\begin{proof} The $\b t $ with $ s f(\b t) = P^{-d}$ have measure $0$, hence, can be ignored. It suffices to show that 
if $ s f(\b t) < P^{-d}$ then   $   \lim_{k} F_k (\b t )=0$. Indeed, for such $\b t $ we shall see that, for $k$ big enough, the range of integration in the definition of $F_k$ is empty.
Indeed, if there is some $y$   in  $F_k$ then     
$$y=y-sf(\b t)+s f(\b t) \leq |y-sf(\b t)|+sf(\b t) \leq \frac{1}{\sqrt k}+s f(\b t),$$ hence, as $k\to\infty$, we get   $y< P^{-d}$. This contradicts the definiton of $F_k$, which includes the condition 
  $y\in[P^{-d},b]$. 
\end{proof}

\begin{lemma}\label{lecocoreliada} Let $\b t \in \c B$. If $s f(\b t)>P^{-d}$ then   $   \displaystyle\lim_{k\to \infty} F_k (\b t ) = \omega(sP^d  f(\b t))$. \end{lemma}\begin{proof} 
For each $y$ such that $|y-sf(\b t)|\leqslant 1/\sqrt{k}$ we have  
$$
 -\frac{1}{\sqrt k}+P^{-d} 
\leq 
-|y-s f(\b t )|+P^{-d} <
 y-s f(\b t )+s f(\b t )\leq |y-s f(\b t )|+b \leq \frac{1}{\sqrt k}+b,$$ 
thus, as $k\to \infty$, we can ignore the condition $y\in [P^{-d},b]$ in the definition of $F_k$. Hence, $$ \lim_{k\to\infty} F_k (\b t )= \lim_{k\to\infty}
\int_{  | f (\b t ) -sy |\leq 1/\sqrt k   } \omega( P^d y) 
 \phi_k   (  f (\b t )-sy )
 \mathrm d  y.$$
Since $\omega$ is continuous and $\phi_k$ has exponential decay, we infer that $$ \lim_{k\to \infty} F_k (\b t )=  
 \omega( P^d s  f(\b t )) 
\lim_{k\to \infty}
\int_{| f (\b t ) -sy | \leq 1/\sqrt k  }  
 \phi_k   (  f(\b t )-sy )
 \mathrm d  y.$$ It now suffices to prove that the last limit over $k\to \infty $ equals $1$. 
By \eqref{eq:ernm} we can dispose of the condition $| f (\b t ) -s y| \leq 1/\sqrt k$
 as $k\to +\infty$ and we can then use the   fact   $\displaystyle\int_{\R }  
 \phi_k  (   \mu)
 \mathrm d   \mu
=1. 
$ \end{proof}
Bringing together Lemmas~\ref{lesds8ghsd6d}, \ref{lecorelid6d} and \ref{lecocoreliada} yields 
$$H= \int\limits_{\substack{ \b t \in \c B \\ s f(\b t) > P^{-d} }}
  \omega(P^d s f(\b t))\mathrm d \b t  .$$ Injecting this into the main term present in Lemma~\ref{lem:guitar} concludes the proof of
Theorem~\ref{lem:vachms}.

\subsection{Proof of Theorem~\ref{pergolesisalvereginainaminor}}
\label{ss:prf15}
We first prove that if 
$ \rho(a,q) \geq 0$ for all $a,q$ then  
 $\displaystyle\sup_{q\in \N}\sum_{a=1}^q |\rho(a,q)|<\infty$.
 Thus, to complete the proof it suffices to 
 work  under the latter assumption only.
 We write 
$$ \sum_{m\leq x}k(m) =\sum_{a\in \Z/q\Z}  \sum_{\substack{ m\leq x\\m\equiv a \md q }} k(m)
$$ and divide through by $\displaystyle\int_1^x\omega(t) \mathrm dt$. 
Letting 
$x\to\infty$  and using 
the assumption $$
\lim_{x\to +\infty}\frac{E(x,q)}{\displaystyle \int_1^x\omega(t) \mathrm dt}= 0
$$
we deduce that for all $a$ and $g$ we have 
\begin{equation}
\displaystyle\sum_{a\in \Z/q\Z}  \rho(a,q)=\lim_{x\to +\infty}\frac{\displaystyle \sum_{m\leq x}k(m)}{\displaystyle \int_1^x\omega(t) \mathrm dt}.
\label{ind}
\end{equation}
Thus,  if $\rho(a,q)\geq 0 $ for all $a,q$ then  
$\displaystyle\sum_{a=1}^q |\rho(a,q)|$ is uniformly bounded for all $q$.

Next we   prove that the limit in~\eqref{eq:george_cuoffee}
exists. Fixing the value of $f(\b t )$ we see that 
$$T_z^{-n+1}  \sum_{\b t \in (\Z/T_z\Z )^{n} } \rho( s f(\b t  ) ,T_z) 
=  \sum_{\nu=1}^{T_z} \rho(\nu,T_z) \prod_{p\leq z}\frac{\#\left\{\b t \in (\Z/p^{t_p(z)}\Z )^{n}: s f(\b t  ) \equiv \nu \md{p^{t_p(z)}} \right\}}{p^{t_p(z)(n-1)}}
,$$ which, by Lemma~\ref{lem:cramer} equals $$ \sum_{\nu=1}^{T_z} \rho(\nu,T_z)  \left(\mathfrak S(\nu)+ o_{z\to\infty}(1) \right).$$
By   assumption, $\displaystyle\sum_{\nu=1}^{T_z} |\rho(\nu,T_z)|$ is bounded independently of $z$, hence, we obtain 
$$ \sum_{\nu=1}^{T_z} \rho(\nu,T_z)   \mathfrak S(\nu)+o_{z\to\infty}(1) = 
\sum_{q=1}^\infty q^{-n}\sum_{\substack{a\in \Z/q\Z \\ \gcd(a,q)=1 }} S_{a,q}
\sum_{\nu=1}^{T_z} \rho(\nu,T_z)   \mathrm e(-a\nu/q) +o_{z\to\infty}(1),$$ where we used the equation before~\cite[Lemma 5.5]{MR0150129}
to express $\mathfrak S$ via $S_{a,q}$ and the bound~\eqref{eq:kotozimi12} to change order of summation. The new series over $q$ converges absolutely
due to~\eqref{eq:kotozimi12} and $$\left| \sum_{\nu=1}^{T_z} \rho(\nu,T_z)   \mathrm e(-a\nu/q) \right| \leq  \sum_{\nu=1}^{T_z} |\rho(\nu,T_z)  |<\infty.$$
 This proves that the limit defining $\sigma(f)$ exists and is independent of the choice of $t_p(z)$ due to \ref{ind}.

The remaining claims 
in Theorem~\ref{pergolesisalvereginainaminor}
will be deduced from Theorem~\ref{lem:vachms}
with the following choice for $m_p(z)$.
 For any prime $p$,   define    $m_p(z)$ as the largest integer satisfying 
 $p^{m_p(z)} \leq z $ and let $W_z=\displaystyle\prod_{p\leq z} p^{m_p(z)}$. 
 For a fixed integer $k\geq 1$ note that the primes $p$
 for which $m_p(z)=k$ are exactly those   in the interval 
 $(z^{1/(k+1)}, z^{1/k}]$.
Summing over all integers $m$ in that interval gives 
\begin{align*}
\widetilde \epsilon(z)&\leq 
\sum_{1\leq k\leq (\log z)/(\log 2)} 
\sum_{ z^{1/(k+1)}<m\leq z^{1/k}} m^{-k-1}
\leq 
\sum_{1\leq k\leq (\log z)/(\log 2)} \int_{ z^{1/(k+1)} }^\infty  \frac{\mathrm dt}{t^{k+1}} 
\\  & \ll  \sum_{1\leq k\leq (\log z)/(\log 2)} 
 \frac{ z^{- k/(k+1)}}{ k }
 \ll \sum_{  k\leq (\log z)/(\log 2)}  \frac{1}{\sqrt z}
 \ll \frac{\log z}{\sqrt z} ,\end{align*} hence, $\widetilde{\epsilon}(z)=o(1)$.
 Since $m_p(z)=[\log z/\log p]$ we infer that $m_p(z)\to\infty$ as $z\to\infty $ and $p$ is fixed.
Fix any $\epsilon>0$. Then there exists $z$ such that 
$$\left| \displaystyle\sigma(f) -W_z^{-n+1}  \sum_{\b t \in (\Z/W_z\Z )^{n} } \rho( s f(\b t  ) ,W_z)\right|<\epsilon \ \ \textrm{ and } \ \  
|\widetilde{\epsilon}(z)+z^{-c}| <\epsilon.$$ Hence, by Theorem~\ref{lem:vachms} we obtain 
  \begin{align*}
&\Bigg|\frac{1}{P^{n} \displaystyle\int \omega( P^d s  f(\b t ))  \mathrm d \b t } 
 \sum_{\substack{ \b t \in \Z^n\cap P\c B \\  s f(\b t ) > 0 }}  k(s f(\b t) ) 
-\sigma(f)\Bigg|
\\ &\ll \epsilon+  \frac{\|k\|_1 }{P^{d} \displaystyle\int \omega( P^d s  f(\b t ))  \mathrm d \b t }   
(P^{-\delta}+ \epsilon  )+ W_z \frac{E(bP^d, W_z )  }{P^{d} \displaystyle\int \omega( P^d s  f(\b t ))  \mathrm d \b t }  ,
\end{align*} where the integral is over $\b t \in \c B$ with $sf(\b t )>P^{-d}$ and the 
implied constant depends only on $f$ and $k$. 
By the last two assumptions in the theorem the error term becomes 
$$
\ll \epsilon+ 
 P^{-\delta}+ \epsilon    + W_z \frac{E(bP^d, W_z )  }{P^{d} \displaystyle\int_1^{bP^d} \omega(t)\mathrm dt  }
 .
 $$
Using the third assumption shows that for $P$ big enough, we have
$$ 
\Bigg| \frac{1}{P^{n}\displaystyle \int \omega( P^d s  f(\b t ))  \mathrm d \b t } 
 \sum_{\substack{ \b t \in \Z^n\cap P\c B \\  s f(\b t ) > 0 }}  k(s f(\b t) ) 
-\sigma(f)\Bigg| \ll \epsilon,$$ where the implied constant depends at most on $f$ and $k$.
Taking arbitrary small $\epsilon$ concludes the proof. 
 \section{Applications}

\subsection{The proof of Theorem~\ref{thm:chowla}}
\label{ss:claa}
 For   fixed $B>0$ Davenport~\cite[Lemma 6]{davedave} proved that  
$$\sum_{\substack{ 1\leq m \leq x \\ m\equiv a \md q  }} \mu(m) \ll \frac{x}{(\log x)^B}
$$ for all $x\geq 1$, $q\in \N$ and $a \in \Z/q\Z$, where the implied constant is independent of $x,a$ and $q$. Letting  $ \rho( a , q)=0$ and 
$\omega( t)=1$ we see that $E(x,q)\ll x (\log x)^{-B}$. One can now easily verify all   assumptions of 
Theorem~\ref{pergolesisalvereginainaminor}, which suffices for the proof.

\subsection{The proof of Theorem~\ref{thm:chowla2}}
\label{ss:claa2}
Recall   the Wilton-type 
bound 
$$\sup_{\alpha \in \mathbb R} \left | \sum_{m\leq x} \lambda_f(m) \mathrm e(\alpha m)
\right| \ll_f \sqrt{ x} \log x$$
that is proved in~\cite[Theorem 5.3]{iwiw}.
Using additive characters modulo $q$ we obtain 
$$\sup_{a\md q } \left | \sum_{\substack{ m\leq x\\ m\equiv a \md q }}
\lambda_f(m)  
\right| \ll_f \sqrt{ x} \log x.$$
The proof now follows similarly as in \S\ref{ss:claa}

\subsection{The proof of Theorem~\ref{thm:divisorddd}}
\label{s:dvsrcrp}
By~\cite[Theorem 1.1]{MR3352438}  one has for all $x\geq 1$ and $a,q \in \mathbb N$ and fixed $\epsilon>0$ 
that 
$$\sum_{\substack{ m\leq x \\ m\equiv a \md q  }} \tau (m) = \frac{x}{q}\sum_{r\mid q } \frac{c_r(a)}{r}\left(\log   \frac{x}{r^2}+ 2\gamma -1 \r)
+O_\epsilon\left( (x^{1/3}+q^{1/2} ) x^{\epsilon}\right)
,$$ 
where $\gamma$ is   Euler's constant, $c_r(a)$ is Ramanujan's sum and the implied constant depends at most on $\epsilon$.
We   apply Theorem~\ref{pergolesisalvereginainaminor},
hence, we do not need to record the dependence of the error term 
on  $a,q$.
The right hand side is   $\rho(a,q) x \log x  +O(x)$,
where 
 $q\rho(a,q)=  \sum_{r\mid q } c_r(a)/r$ and the implied constant depends at most on $q$ and $a$. 
Dividing through by $x\log x$ and letting $x\to\infty$ 
 shows that $\rho(a,q) \geq 0 $.
Letting   $\omega(t ) = \log t $   one has 
\[E(x,q)=\sum_{\substack{ m\leq x \\ m\equiv a \md q  }} \tau(m) -\rho(a,q) \int_1^x \omega(t) \mathrm d t  =O(x).\]
It is now easy to verify all remaining assumptions of Theorem~\ref{pergolesisalvereginainaminor}. 
 The real density in the main term is 
\[ \int\limits_{\substack{ \b t \in \c B \\ s f(\b t ) > P^{-d} } }\log\left( P^d s f(\b t)  \right)  \mathrm d \b t  =d  (\log P)\mathrm{vol}(\b t \in \c B: s f(\b t)>P^{-d})+O(1) .\]
This is asymptotic to $d  (\log P)\mathrm{vol}(\b t \in \c B: s f(\b t)>0)$ since 
$\mathrm{vol}(\b t \in \c B: | f(\b t)|\ll P^{-d})\ll P^{-d}$ by~\cite[Lemma 1.19]{brobook}.

It remains to study $\sigma(f)$. Since $c_r(a)$ is multiplicative with respect to $r$ we   write   $  \sigma(f)$ as $$\prod_{p=2}^\infty
\lim_{m\to\infty}
p^{-(n-1)m} \sum_{\b t \in (\Z/p^{m}\Z )^{n} }  \rho( s f(\b t  ) ,p^{m} )  
=\prod_{p=2}^\infty
\lim_{m\to\infty}
p^{-nm} \sum_{\b t \in (\Z/p^{m}\Z )^{n} } \sum_{r=0}^{m} p^{-r}  c_{p^r}(s f(\b t  ))
.$$ Using the fact that $c_{p^r}(a)=c_{p^r}(b)$ when $a\equiv b \md {p^r}$ we obtain 
 \begin{equation} \sum_{r=0}^{m} p^{-r(n+1) } \sum_{\b x \in (\Z/p^{r}\Z )^{n} } c_{p^r}(s f(\b x  ))  
=\sum_{r=0}^{m} p^{-r(n+1) }  \sum_{\substack{ a\in \Z/p^r \Z\\p\nmid a  }}
S_{a,p^r}.
\label{eqq}
\end{equation}
 By~\eqref{eq:touselater} we have the following for each $m\geq 1 $, 
$$
 \sum_{ 0\leq r \leq m  }  p^{-rn} \sum_{\substack{  a \in \Z/p^r\Z\\ p\nmid  a  } }S_{ a,p^r} 
=\frac{\#\{\b t \in (\Z/p^{m } )^{n}:  f(\b t ) = 0  \}}{p^{(n-1) m}}=:\Gamma_m
$$from which it immediately follows that for all $m\geq 1 $ one has 
$$p^{-mn} \sum_{\substack{  a \in \Z/p^m\Z\\ p\nmid  a  } }S_{ a,p^m} 
=\Gamma_m-\Gamma_{m-1}.$$
Therefore, (\ref{eqq}) is equal to \begin{align*}
&1+\sum_{r=1}^{m } p^{-r} (\Gamma_r-\Gamma_{r-1})
=\frac{\Gamma_m}{p^{m+1}}+\left(1-\frac{1}{p}\r)\sum_{k=0}^m \frac{\Gamma_k}{p^k}
\\=&
\frac{\Gamma_m}{p^{m+1}}
+\left(1-\frac{1}{p}\r)\sum_{k=0}^m \frac{\#\{\b t \in (\Z/p^{k } )^{n}:  f(\b t ) = 0  \}}{p^{n k}}
\\=&
\frac{\Gamma_m}{p^{m+1}}
+\left(1-\frac{1}{p}\r)\frac{1}{p^{n m}}
\sum_{\b t \in (\Z/p^{m } )^{n}} \sum_{k=0}^m  \mathds 1(p^k\mid f(\b t )   )
.\end{align*}
Recalling the definition of $ \tau_{p^m}(n)$ and taking the limit as $m\to\infty$,
this becomes 
$$\left(1-\frac{1}{p}\r)\lim_{m\to\infty} \frac{1}{p^{n m}}
\sum_{\b t \in (\Z/p^{m } )^{n}} \tau_{p^m}(f(\b t )),$$
which concludes the proof of Theorem~\ref{thm:divisorddd}.

\subsection{A shifted convolution problem in arithmetic progressions}
\label{ss:convoltn}

Let $\chi$ be the non-principal Dirichlet character modulo $4$ so that 
Dirichlet convolution $1\ast \chi(n)$ equals $4$ times the number of representations of $n$ as a sum of two integer squares for any integer $n$. For   $a,q\in \N$ we 
 shall estimate
\beq{eqnabalwmousiki}{ S(x;q,a):=
\sum_{\substack{ m\leq x \\ m\equiv a \md q }} r(m )  r(m+1)}
under the conditions 
\begin{equation}
    \label{eq:assummptns}
4\mid q, \ \ 
 v_p(a)\leq v_p(q)-1 \ \forall p\mid q
.\end{equation} The case $q=1$ can be treated using a variety of methods,
for example, the first proof 
of Estermann~\cite{ester} 
uses Kloosterman sums
while the proofs 
in~\cite{MR1661040} 
and~\cite[Corollary 15.12, part (2)]{iwan} use 
spectral methods. We were unable to locate a
reference that provides an explicit error term for general 
arithmetic progressions and we therefore give a proof here.
\begin{theorem}
\label{thm:thankyouthankyouthankyou} For all $a,q\in \N,x\geq 1 $ 
satisfying $q\leq x^{1/2}$
and~\eqref{eq:assummptns} we have 
$$
S(x;q,a)
=  \pi^2  x\frac{\eta_q(a)\eta_q(a+1)}{q^3}
\mathds 1(a\equiv 0,1 \md 4 )
 \prod_{p\nmid q }\left(1-\frac{1}{p^2}\right)
+O\left(x^{5/6} (\log x)^{20} q^{3/2}\tau(q)^4 \right),$$ where the implied constant is absolute
and for any $b$ integer \begin{equation}
\eta_q(b):=
 \#\left\{\b y \in (\Z/q\Z)^2: 
y_1^2+y_2^2\equiv  b \md q  \right\} .
\label{eta}
\end{equation}
\end{theorem}

As a first step we recall that $r(m)$ vanishes for $m\equiv 3 \md 4$, thus, 
\beq{eqnabadefnn}{ S(x;q,a)= S_0(x;q,a)+ S_1(x;q,a),} where 
$$ S_i(x;q,a):=
\sum_{\substack{ m\leq x, m\equiv i \md 4 \\ m\equiv a \md q  }} r(m )  r(m+1)
.$$ 
Indeed, either $m$ is odd and for $r(m)r(m+1)$ not to vanish, $m\equiv 1 \md{4}$ or $m+1$ is odd and for $r(m)r(m+1)$ not to vanish, $m\equiv 0 \md{4}$.
By Dirichlet's hyperbola trick
when $k$ is odd we have 
$$\frac{1}{4} r(k)=
\sum_{d e= k}\chi(d)=
\sum_{\substack{ d\mid k \\d\leq \sqrt x }} \chi(d)
+\chi(k)
\sum_{\substack{ e\mid k \\e<k/\sqrt x}} \chi(e)
.$$ This shows that 
\beq{eqdiridirifnn}{ S_1(x;q,a)= 4S_1^{-}(x;q,a)+ 4S_1^{+}(x;q,a),}
where 
$$
S_1^{-}(x;q,a):=\sum_{d\leq \sqrt x} \chi(d) 
\sum_{\substack{ m\leq x, m\equiv 1 \md 4 
\\ m\equiv a \md q, d\mid m }} r(m+1) $$ and $$
S_1^{+}(x;q,a):= \sum_{e\leq \sqrt x} \chi(e) 
\sum_{\substack{ m\in (e\sqrt x,x], m\equiv 1 \md 4 
\\ m\equiv a \md q, e\mid m }}   r(m+1) 
.$$  Since $4 \mid q $ by~\eqref{eq:assummptns}
we infer that 
the three congruences in the sums $S_1^-, S_1^+$  are soluble 
if and only if   
 $4\mid a-1$ and $ \gcd(d,q)\mid a.$ Denoting the least common multiple by $[\cdot,\cdot]$ we see that under the aforementioned conditions there exists a unique   
 $t\in \Z/[d,q]\Z$ such that 
\begin{equation}\label{eq:light}
 m\equiv 1 \md 4, m\equiv a \md q, m\equiv 0 \md d
 \iff 
 m\equiv t \md{[d,q]}
 .\end{equation}
Hence,  
$$S_1^-(x;q,a)=\mathds 1(a\equiv 1 \md 4  ) 
\sum_{\substack{ d\leq \sqrt x \\ \gcd(d,q)\mid a }} \chi(d)
\sum_{\substack{ m\leq x   \\ m\equiv t\md{[d,q]}  }} r(m+1 )   
.$$ \begin{lemma}\label{lem:bulgaria}
For all $a,q$ satisfying~\eqref{eq:assummptns}  and all $x\geq 1 $ we have 
$$S_1^-(x;q,a)
=\mathds 1(a\equiv 1 \md 4  ) \pi x \mathcal M+ 
+O\left(q  \tau(q)^4 x^{5/6} (\log x)^{19}\right),
$$ where the implied constant is absolute and 
$$\mathcal M :=\sum_{\substack{ d\leq \sqrt x \\ \gcd(d,q)\mid a }} \chi(d)
\frac{\#\left\{\b y \in (\Z/[d,q]\Z)^2: 
y_1^2+y_2^2\equiv a+1 \md q , 
y_1^2+y_2^2\equiv 1 \md d \right\}}{[d,q]^2}.$$\end{lemma}
\begin{proof}Applying~\cite[Theorem, page 262]{tolev} gives
the error term 
$$
\ll 
 \sum_{\substack{ d\leq \sqrt x   }}  
\left( \left([d,q]^{1/2}+x^{1/3}\right) \gcd(a,[d,q])^{1/2} 
\tau([d,q])^4(\log x)^4 \r) .$$ By~\eqref{eq:assummptns} we 
have $\gcd(a,[d,q])\leq q$. Indeed, if a prime $p\mid a,d$ but $p\nmid q$, since $d\mid m$, we deduce that $p\mid q$ which is a contradiction. Hence, the error term is   $$ 
\ll q^{1/2} \tau(q)^4 (\log x)^4
\sum_{\substack{ d\leq \sqrt x   }}  
  \left((dq)^{1/2}+x^{1/3}\right)  \tau(d)^4   
 \ll q  \tau(q)^4 x^{5/6} (\log x)^{19} .$$
 The main term supplied by ~\cite[Theorem, page 262]{tolev}
equals $$\mathds 1(a\equiv 1 \md 4  ) \pi x 
\sum_{\substack{ d\leq \sqrt x \\ \gcd(d,q)\mid a }} \chi(d)
\frac{\eta_{t+1}([d,q])}{[d,q]^2},$$
 which takes the required shape in light of~\eqref{eq:light} and the fact that the congruence modulo 4 is implied by the one modulo $q$ provided that $a\equiv 1 \md{4}.$
\end{proof} One can similarly prove the estimate \beq{eq:simvivaldi}{
 S_1^+(x;q,a)
=\mathds 1(a\equiv 1 \md 4  ) \pi \left( x \c M  -E \sqrt{x} \r) 
+O\left(q  \tau(q)^4 x^{5/6} (\log x)^{19}\right),
} where the implied constant is absolute and 
$$E:=\sum_{\substack{ e\leq \sqrt x \\ \gcd(e,q)\mid a }} \chi(e) e
\frac{\#\left\{\b y \in (\Z/[e,q]\Z)^2: 
y_1^2+y_2^2\equiv a+1 \md q , 
y_1^2+y_2^2\equiv 1 \md d \right\}}{[e,q]^2}.$$Taking into account the oscillation of $\chi(e)$, we next show that $E$ makes a negligible contribution.
\begin{lemma}\label{lem:barkingdog} 
For all $a,q$ satisfying~\eqref{eq:assummptns}  and all $x\geq 1 $ we have 
$$E\ll q^{3/2} (\log q) \tau(q)  (\log x)  ,$$ where the implied constant is absolute.\end{lemma}
\begin{proof} Write $e=e_0 e_1$ where $e_1$ is coprime to $q$ and each prime factor 
of $e_0$ divides $q$. Then the condition $\gcd(e,q)\mid a$ becomes 
$\gcd(e_0,q)\mid a$ and by~\eqref{eq:assummptns} 
we infer that $ e_0\mid q$.
Thus,  
$$ E=\sum_{\substack{ e_0 \leq \sqrt x \\ e_0\mid \gcd(a,q)  
 }} 
\chi(e_0) e_0
\frac{\#\left\{\b y \in (\Z/q\Z)^2: 
y_1^2+y_2^2\equiv a+1 \md q  \right\}}{q^2}
E_1( \sqrt{x} /e_0),$$ where $$
E_1(T):=
\sum_{\substack{ e_1 \leq T \\ \gcd(e_1,q)=1 }} \chi(e_1)  
\frac{\#\left\{\b y \in (\Z/e_1\Z)^2:  
y_1^2+y_2^2\equiv 1 \md{e_1} \right\}}{e_1}
.$$ By~\cite[pages 27--28]{rogier} we infer that $E_1(T)$ equals 
$$\sum_{\substack{ e_1 \leq  T \\ \gcd(e_1,q)=1 }} \chi(e_1)  \prod_{p\mid e_1} \left (1-\frac{\chi(p)}{p}\right) 
=\sum_{\substack{ k\leq T\\ \gcd(k,q)=1}} \frac{\mu(k)\chi(k)}{k}
\sum_{\substack{ e_1\leq T, k\mid e_1  \\ \gcd(e_1,q)=1}}\chi(e_1)
=\sum_{\substack{ k\leq T\\ \gcd(k,q)=1}} \frac{\mu(k)}{k}
\sum_{\substack{ t\leq T/k  \\ \gcd(t,q)=1}}\chi(t)
.$$ We may view $\mathds 1(\gcd(t,q)=1)\chi(t)$ as a character modulo $q$ (remember here that $4 \mid q$)
in $t$. Hence, the P\'olya--Vinogradov estimate gives the bound $
E_1(T)\ll  \sqrt q (\log q)(\log T)  $. Hence, 
$$E\ll \sqrt q (\log q)(\log x) \sum_{e_0\mid \gcd(a,q)} e_0
\ll q^{3/2} (\log q)(\log x) \tau(q) $$ since $\gcd(a,q)\mid q$.
\end{proof}

\begin{lemma} 
\label{lem:thankuroger} Assume that $q\in \N$ is divisible by $4$, that 
$a\equiv 1 \md 4$  and that $a\in \Z/q\Z$ is such that $v_p(a)<v_p(q)$ for all $p\mid q$.
Then $$\sum_{d_0\mid \gcd(a,q)}\chi(d_0)= \frac{1}{2}
\frac{\eta_q(a)}{q} \prod_{p\mid q} \left(1-\frac{\chi(p)}{p}\right)^{-1},$$
where $\eta_q(a)$ is defined in (\ref{eta}).
\end{lemma}
\begin{proof}
 The sum over $d_0$ can be written as 
 $$
 \prod_{\substack{p\equiv 1 \md 4  \\ p\mid \gcd(a,q)}} (1+v_p( a) )
 \prod_{\substack{p\equiv 3 \md 4  \\ p\mid \gcd(a,q)}} \frac{1+(-1)^{v_p(a)}}{2}
. $$ Now we can replace the condition $p\mid \gcd(a,q)$ by $p\mid q$ in both 
products because if $v_p(a)=0$ then the analogous terms in the products equal $1$.
By~\cite[pages 27--28]{rogier} and the fact that $v_p(a)<v_p(q)$
we then obtain 
 $$
 \prod_{\substack{p\equiv 1 \md 4  \\ p\mid q}} \frac{\eta_{p^{v_p(q)}}(a)}
 {p^{v_p(q)}(1-1/p)}
 \prod_{\substack{p\equiv 3 \md 4  \\ p\mid q}} \frac{\eta_{p^{v_p(q)}}(a)}
 {p^{v_p(q)}(1+1/p)}. $$
Lastly, by~\cite[Equation (8.4)]{rogier} and our assumptions
$4\mid q$ and $a\equiv 1 \md 4$ this becomes 
 $$ \frac{1}{2}
 \prod_{\substack{ p\mid q}} \frac{\eta_{p^{v_p(q)}}(a)}
 {p^{v_p(q)}(1-\chi(p)/p)}
 ,$$ which completes the proof.
 \end{proof}

\begin{lemma}\label{lem:pergolesisimfonias}
For all $a,q$ satisfying~\eqref{eq:assummptns}  and all $q\geq x^{1/2},$ we have 
$$ \c M = \frac{\pi}{8}
\frac{\eta_q(a) \eta_q(a+1)}{q}  
\prod_{p\nmid q} \left(1-\frac{1}{p^2}\right)
+O\left( \tau(q) q^{3/2}\frac{(\log x)^2}{\sqrt x}\right)
,$$ where the implied constant is absolute.\end{lemma}
\begin{proof} Writing $d=d_0d_1$ as in the proof of Lemma~\ref{lem:barkingdog} 
we see that $\c M$
equals $$\frac{\#\left\{\b y \in (\Z/q\Z)^2: 
y_1^2+y_2^2\equiv a+1 \md q  \right\}}{q^2}
\sum_{\substack{   d_0\mid \gcd(a,q)  }} 
\chi(d_0) 
\sum_{\substack{ d_1\leq \sqrt x/d_0 \\  \gcd(d_1,q)=1 }} 
\frac{\chi(d_1)}{d_1}
\prod_{p\mid d_1} \left (1-\frac{\chi(p)}{p}\right)  
.$$ We   used~\cite[pages 27--28]{rogier} and the fact that 
 $y_1^2+y_2^2\equiv 1 \md {d_0}$ is implied by 
$y_1^2+y_2^2\equiv a+1 \md q $ due to $d_0\mid \gcd(a,q)$.
Note that the condition $d_0\leq \sqrt x$ is implied by $d_0\mid q$ owing 
to the assumption $q \geqslant x^{1/2}$.  
The sum over $d_1$ equals
$$\sum_{\substack{ k\leq \sqrt x/d_0\\\gcd(k,q)=1 }} \frac{\mu(k)}{k^2}
\sum_{\substack{ t \leq \sqrt x/(kd_0) \\ \gcd(t,q)=1 }} \frac{\chi(t)}{t}
.$$ 
The sum over $t$ can be seen as converging to the value of an $L$-function at $1$ with Dirichlet character $t \mapsto \chi(t) \mathds 1(\gcd(t,q)=1)$. Hence, by \cite[Lemma 16]{Schmidt} for example, it equals 
$$\sum_{\substack{ t =1 \\ \gcd(t,q)=1 }}^\infty 
\frac{\chi(t)}{t}+O\left(q^{1/2}\log(q)\frac{kd_0}{\sqrt x}\right),$$ 
with an absolute implied constant. Thus, the main term becomes 
$$
\begin{aligned}
\sum_{\substack{ k\leq \sqrt x/d_0\\\gcd(k,q)=1 }} \frac{\mu(k)}{k^2}
\left( \sum_{\substack{ t=1 \\ \gcd(t,q)=1 }} ^\infty \frac{\chi(t)}{t}
+O\left(q^{1/2}\log(q)\frac{ k d_0}{\sqrt x} \right) \right)=
&\prod_{p\nmid q} \left(1-\frac{1}{p^2}\right)
\sum_{\substack{ t=1 \\ \gcd(t,q)=1 }} ^\infty \frac{\chi(t)}{t}\\
&
+O\left(q^{1/2}\log(q) \frac{d_0}{\sqrt x}\log x\right),
\end{aligned}$$ where we used the standard bound $L(1,\psi)\ll \log r$ for a 
non-principal character $\psi$ modulo $r$. Finally, we use Lemma~\ref{lem:thankuroger} to deal with the sum over $d_0$ and the calculation $$
\prod_{p\mid q}\left(1-\frac{\chi(p)}{p}\right)^{-1} \sum_{\substack{t=1\\ \gcd(t,q)=1}}^\infty \frac{\chi(t)}{t}=
\sum_{ t=1 }^\infty\frac{\chi(t)}{t}=
\frac{\pi}{4}.$$
\end{proof}
\subsection{The proof of Theorem~\ref{thm:thankyouthankyouthankyou}}
\label{ss:finishingproof}
Injecting Lemma~\ref{lem:barkingdog} into~\eqref{eq:simvivaldi}
gives an asymptotic for $S_1^+$ in terms of $\mathcal{M}$.
By Lemma~\ref{lem:bulgaria} we have a similar asymptotic for  $S_1^-$.
Putting these estimates into~\eqref{eqdiridirifnn}
and then using Lemma~\ref{lem:pergolesisimfonias} to estimate $\mathcal M$
yields an asymptotic for $S_1$.
A similar argument works in an identical manner for $S_0$ and that 
  concludes the proof of Theorem~\ref{thm:thankyouthankyouthankyou}.

\subsection{The proof of Theorem~\ref{thm:analytic_HP}}
\label{ss:proof_thrm_convol}
We shall need the following   lemma.
\begin{lemma}
\label{lem:nairlem} For every prime $p$,
 any integer $m\geq 1$ and any $x\geq 1$ with 
$p^m \leq  x^{1/4}$ we have
$$\sum_{\substack{ |n|\leq x  \\ p^m \mid n }} r(n)r(n+1)\ll 
 \frac{m  }{ p^m }   x,$$ where the implied constant is absolute.
\end{lemma}
\begin{proof} We have
$$
\sum_{\substack{ |n|\leq x  \\ p^m \mid n }} r(n)r(n+1) \ll \sum_{\substack{ |n|\leq x  \\ p^m \mid n }} r(n(n+1))
$$
where we used the coprimality of $n,1+n$
to bound $r(n)r(1+n)\ll r(n (1+n))$. Note that we must have $n\equiv 0,1 \md{4}$ in order to have $r(n)r(n+1)\neq 0.$
The result would follow  from the work of 
Nair~\cite{nnair}, however, its application is 
prohibited by the fact that  
 the polynomial $n(n+1)$ has $2$ as  a fixed divisor.
 It   suffices  to work with the 
multiplicative function $r'=r/4$. 
Then the sum in the lemma is 
$$ \ll  \sum_{\substack{ |t|\leq x/p^{m}   }} 
r'(p^m t) r'(p^m t+1)
\leq (m+1)
\sum_{\substack{ |t|\leq x/p^{m}   }} 
r'(t (p^m t+1))
,$$ since  $r'(ab)\leq \tau(a) r'(b)$ holds 
for all integers $a,b$.
Using the bound $r'(s)\leq \tau(s) 
\ll s^{1/10}$ we see that 
the  contribution of the terms 
for which $2^k \mid t$ for some $k$ satisfying 
$2^k>x^{1/4}$ is $$\ll  (m+1)
\sum_{\substack{ |t|\leq x/p^{m} \\ 2^ k \mid t   }} 
x^{1/5} \ll m x^{1/5} \left(\frac{x}{p^{m}2^{k}}+1\right)
\ll m  x^{19/20} p^{-m} + m x^{1/5}.
$$ The assumption 
$p^m\leq x^{1/4}$ implies that $x^{1/5}\leq xp^{-m}$,
hence, the bound is satisfactory.

When $k\geq 1$ and $2^k \leq x^{1/4}$
  we can write, in the case  $t=2^kt'$ with $t'$ odd
to obtain   
$$ \ll m \sum_{1\leq k \leq  ( \log x)/(4\log 2)} k
\sum_{\substack{ |t'|\leq x/( p^{m} 2^{k} ) \\ 2\nmid t' }} 
r'(  t' (p^m 2^k t'+1)) $$
where we used $r'(2^k s) \ll k r'(s)$. 
If $t'\equiv 3 \md 4$ then $r'(t')=0$, hence, 
we may assume that $t'\equiv 1 \md 4$.
Then we can write $t'=4s+1$, thus, we get the bound 
 $$ \ll m \sum_{0\leq k \leq  ( \log x)/(4\log 2)} k
\sum_{\substack{ |s'|\leq 4x/( p^{m} 2^{k} )   }} 
r'( P(s) ) ,$$
where $P(s)= (4s+1)  (p^m 2^k (4s+1) +1)$. Since $P(0)$ is odd,
we can apply Nair's result~\cite[Theorem, page 259]{nnair} with $\delta=1$. Indeed, 
we have $\| P\|  \ll p^m 2^k\ll (x p^{-m} 2^{-k})$
because both $2^k$ and $p^m$ are at most $  x^{1/4}$.
We obtain 
$$\sum_{\substack{ |s'|\leq 4x/( p^{m} 2^{k} )   }} 
r'( P(s) ) \ll \frac{x}{ p^{m} 2^{k} } 
\prod_{2<p \leq x} \bigg (1-\frac{2}{p}\bigg) 
\exp\bigg(\sum_{p\leq x} \frac{2 r'(p)}{p}\bigg) \ll \frac{x}{ p^{m} 2^{k} } $$ with an absolute implied constant. 
Thus, the overall contribution is 
 $$ \ll x 
 \frac{m }{ p^{m}  }
 \sum_{0\leq k \leq  ( \log x)/(4\log 2)}  
\frac{k}{   2^{k} } ,$$
which is sufficient. The contribution of the cases with $k=0$ 
can be dealt with in an analogous way by making substitutions in 
the term $r'(n+1)$.
  \end{proof}
The proof
of Theorem~\ref{thm:analytic_HP}
is an application of
Theorem~\ref{lem:vachms} with 
$$m_p(z)=1+\left [\frac{\log z}{\log p}\right],\quad 
k(m)=r(m)r(m+1) \mathds 1(p\leq z \Rightarrow v_p(m)<m_p(z)),$$
together with $s=1$,
$\omega(t)=\pi^2$ and 
$$
\rho(a,q)=\frac{\eta_q(a)\eta_q(a+1)}{q^3}
\mathds 1(a\equiv 0,1 \md 4 )
 \prod_{p\nmid q }\left(1-\frac{1}{p^2}\right)
$$
in~\eqref{eq:basicproperty}.
We have $W_z\leq \mathrm e^{3z}$ since 
$$\log 
W_z=\sum_{p\leq z} \left(1+
\left [\frac{\log z}{\log p}\right ] \right) \leq 3z
 $$ for all large enough $z$ by the Prime Number Theorem.
Hence, by Theorem~\ref{thm:thankyouthankyouthankyou} we have 
$$\frac{E(bP^d, W_z ) W_z  }{P^{d}   }
\ll \frac{(\log P)^{19} W_z^{5/2}\tau(W_z)^4  }
{P^{d/6} } \ll  \frac{(\log P)^{19} W_z^3 }
{P^{d/6} } \ll 
 \frac{(\log P)^{19}} { P^{d/12} }
$$ as long as we assume $z=z(P) := \frac{d}{108} \log P$ and using the $q=1$ case.
Thus, the error term in Theorem~\ref{lem:vachms} is 
$$\ll P^{-\delta}+\widetilde \epsilon(z)+z^{-c} +  \frac{(\log P)^{19}} { P^{d/12} } 
\to 0  \textrm{ as }
 P\to +\infty $$ since one can show that $\widetilde \epsilon(z)\to 0$
in a similar manner as in \S \ref{ss:prf15}.
For the main term we note that 
$$\lim_{P\to\infty}
\int\limits_{\substack{ \b t \in \c B: s f(\b t ) > P^{-d}  } }  \omega( P^d s  f(\b t ))  \mathrm d \b t = \pi^2 \mathrm{vol}\left(\{ \b t \in \c B \, : \, f(\b t)> 0\}\right)
.$$ 
Since $$ \prod_{p\nmid T_z }\left(1-\frac{1}{p^2}\right)=
\prod_{p>z}\left(1-\frac{1}{p^2}\right)=1+O\left(\frac{1}{z}\right),
$$
we see that 
$$
\sum_{\b t \in (\Z/T_z\Z )^{n} } \rho(  f(\b t  ) ,T_z) =
\frac{1+O(1/z)}{{T_z}^3}
\sum_{\substack{\b t \in (\Z/T_z\Z )^{n}\\f(\b t )\equiv 0,1 \md 4} } 
\eta_{T_z}(f(\b t ))\eta_{T_z}(f(\b t )+1)
. $$ The condition $f(\b t )\equiv 0,1 \md 4$ is implied by the fact that $4$ divides $T_z$ and the presence of the terms 
$\eta_{T_z}(f(\b t ))\eta_{T_z}(f(\b t )+1)$.
Hence, 
\begin{align*}
&\lim_{z\to \infty}\frac{ \#\{(\b t,\b x,\b y) \in (\Z/T_z\Z )^{n+4}:
x_1^2+x_2^2=f(\b t), y_1^2+y_2^2=1+f(\b t ), v_p(f(\b t ))<m_p(z)\,\,\forall p\leq z\}}
{T_z^{n+2}}
\\ =&\lim_{z\to \infty}
T_z^{-n+1}
\sum_{\substack{ \b t \in (\Z/T_z\Z )^{n} \\ v_p(f(\b t ))<m_p(z) \,\, \forall p\leq z}} 
\rho(  f(\b t  ) ,T_z),
\end{align*} which can be shown to exist as in \S \ref{ss:prf15}
by using the fact that $\rho(f(\b t) ,T_z)\geq 0$.
The condition $v_p(f(\b t ))<m_p(z)$ can be ignored since the number of 
$\b t \in (\Z/p^{m_p(z)}\Z )^{n} $ with $f( \b t )=0$ equals 
$p^{m_p(z)(n-1)}(1+O(p^{-1-c}))$ for some $c>0$ by~\eqref{eq:kotozimi12} and~\eqref{eq:touselater}
and the number of $\b x ,\b y \in \left(\Z/p^{m_p(z)\Z}\right)^4$
with $x_1^2+x_2^2=0, y_1^2+y_2^2=1$ is $O\left(m_p(z)p^{2m_p(z)}\right)$ by~\cite[pages 27--28]{rogier}.
Hence, \begin{align*}
& \prod_{p\leq z}\frac{ \#\{(\b t,\b x,\b y) \in (\Z/p^{m_p(z)}\Z )^{n+4}:
x_1^2+x_2^2=f(\b t), y_1^2+y_2^2=1+f(\b t ), v_p(f(\b t ))<m_p(z)\forall p\leq z\}}
{p^{(n+2)m_p(z)}}
\\ =& V(z) \prod_{p\leq z} \frac{ \#\{(\b t,\b x,\b y) \in (\Z/p^{m_p(z)}\Z )^{n+4}:
x_1^2+x_2^2=f(\b t), y_1^2+y_2^2=1+f(\b t )\}}
{p^{(n+2)m_p(z)}},
\end{align*} 
where $V(z)=\displaystyle\prod_{p\leq z} \left(1+O\left(m_p(z) p^{-m_p(z)}\right)\right)$. 
We   have   $\lim_{z\to \infty } V(z)=1$
since $\log V(z)$ is 
$$\ll \sum_{p\leq z}\frac{m_p(z)}{p^{m_p(z)}}\leq (\log z)\sum_{p\leq z} 
p^{-1-[\log z/\log p]},$$ which was shown to converge to $0$ in \S \ref{ss:prf15}.

We have so far shown that 
$$\lim_{P\to\infty} P^{-n}\sum_{\substack{ \b x \in \Z^n\cap P\c B \\ 
v_p(f(\b x )) <m_p(z) \,\, \forall p\leq z}} 
r(f(\b x )) r(1+f(\b x ) )= \pi^2\mathrm{vol}\left(\{ \b t \in \c B \, : \, f(\b t)> 0\}\right)
\prod_{ p } \sigma_p(f),$$ where 
$ \sigma_p(f)$ are as in Theorem~\ref{thm:analytic_HP} and
it  remains to get  rid of    the condition $v_p(f(\b x )) <m_p(z) $.

 Letting $\c A$ be the set of integers $\nu$ for which there exists $p\leq z$
 with $p^{m_p(z)}\mid \nu$ we have 
   $$   \sum_{\substack{ \b x \in \Z^n\cap P\c B \\ 
f(\b x ) \in \c A }}  r(f(\b x )) r(1+f(\b x ) )\ll 
\sum_{\substack{|\nu|\ll P^d  \\ \nu \in \c A }} r(\nu)r(1+\nu)
   \sum_{\substack{ \b x \in \Z^n\cap P\c B \\ 
f(\b x ) =\nu }}  1.
,  $$
 By~\cite[Lem.5.5]{MR0150129} we can 
bound this by 
$$\ll
P^{n-d} 
\sum_{\substack{|\nu|\ll P^d  \\ \nu \in \c A }} r(\nu) r(1+\nu)
\mathfrak S(  \nu) J(  \nu P^{-d}) 
\ll P^{n-d}
\sum_{\substack{|\nu|\ll P^d  \\ \nu \in \c A }} r(\nu) r(1+\nu)
$$ since the estimates $\mathfrak S(  \nu), J( \nu P^{-d}) =O(1)$ hold 
uniformly in $\nu$ by~\eqref{son234} and~\eqref{bchrand4}.
Since $p^{m_p(z)} \leq z \leq P^{d/4}$
we can employ 
  Lemma~\ref{lem:nairlem}
to  obtain 
$$ 
\sum_{\substack{|\nu|\ll P^d  \\ \nu \in \c A }} r(\nu)r((1+\nu)
\ll P^d 
\sum_{p\leq z} \frac{m_p(z)}{p^{m_p(z)}} 
\leq P^d \frac{\log z}{\log 2}\widetilde \epsilon(z) 
.$$ By the estimate 
$\widetilde \epsilon(z) \ll 
(\log z)/\sqrt z$ proved in \S \ref{ss:prf15}
 we get $\ll P^d (\log z)^2/\sqrt z$, 
 which concludes the proof.
 
\subsection{The proof of Theorem~\ref{pieropanwn}}
\label{ss:pieropnapl}
We shall use Theorem~\ref{pergolesisalvereginainaminor} with $k$ being the indicator function of    
elements  of $\c A$.
By our assumption for $q=1$ we see that   $\displaystyle\int_1^T\omega(t)\mathrm dt$ is asymptotic to $\#\{ \c A\cap [1,T]\}$. By assumption, 
$\c A$ is non-empty, hence $\rho(1,1)\displaystyle \int_1^T\omega(t)\mathrm dt$ is non-zero for all large $T$. This verified the second assumption of Theorem~\ref{pergolesisalvereginainaminor}. Note that $\rho(r,q)$ is non-negative, it being the limit of non-negative counting functions.
The remaining assumptions are easy to verify.

\subsection{The proof of Theorem~\ref{thm:hyperelipt}} 
\label{s:pwerdd}
We apply Theorem~\ref{pergolesisalvereginainaminor} with $s=1$ and $\c A_k$ being the set of 
$k$-th powers of positive integers. For all   $r,q$ and $x$, one has  
$$ \sum_{\substack{ 1\leq m \leq x \\ m\equiv r \md q  } }\mathds 1_{\c A_k}(m)  =
 \sum_{\substack{ 1\leq t \leq x^{1/k} \\ t^k \equiv r \md q  } } 1=
 \sum_{\substack{ y \in \mathbb{Z}/q\mathbb{Z}  \\ y^k \equiv r \md q  } }  \left(\frac{x^{1/k} }{q}+O(1) \r)
=\rho(r,q) \int_1 ^ x \frac{\mathrm d t }{k t^{1-1/k} } + O(1) ,$$
where  $q\rho(r,q)=   \#\{y \md q : y^k =r \} $ and the implied constant is independent of $x$. Recall that we are applying Theorem~\ref{pergolesisalvereginainaminor} and hence only need to work with fixed modulus $q$ and don't need to record the dependence on $q$ in the error terms. 
Thus, letting $\omega(t) = 1/(k t^{1-1/k}) $ we find that $$ \frac{\#\{a\in \c A_k: a\leq x, a\equiv r \md q \}}{\displaystyle\int_1^x \omega(t)\mathrm d t }-
\rho(r,q) \ll \frac{1}{\displaystyle\int_1^x \omega(t)\mathrm d t } \to 0.$$ To verify the remaining assumption of Theorem~\ref{pergolesisalvereginainaminor}, 
we use the assumption that $f$ assumes at least one strictly positive  value in $\c B$ so that for all large $P$ we have 
\beq{leonardoleosalveregina}{P^d \int\limits_{\substack{ \b t \in \c B\\  f(\b t ) > P^{-d} } } \frac{  \mathrm d \b t }{k (P^d    f(\b t )) ^{1-1/k} }
= \frac{  P^{d/k}  }{k  }\int\limits_{\substack{ \b t \in \c B\\  f(\b t ) > P^{-d} } } \frac{  \mathrm d \b t }{     f(\b t ) ^{1-1/k} }
 \gg P^{d/k}  \gg   \int_1^{bP^d}  \frac{\mathrm dt}{ t^{1-1/k} }
.} 
Invoking   Theorem~\ref{pergolesisalvereginainaminor} allows us to conclude the proof.

\subsection{The proof of Theorem~\ref{thm:mful}}
\label{s:fulcrp}
A positive integer $k$ is $m$-full equivalently when it has the shape 
$k= k_1^m k_2^{m+1}\cdots k_{m}^{2m-1} $ where $k_2\cdots k_m$ is square-free. Thus, for any $x,r$ and $q$, we have 
$$ \#\{ k \in \N \cap[1,x]  : k \textrm{ is } m\textrm{-full }, k\equiv r \md q  \} = \sum_{\substack{ k_1^m k_2^{m+1}\cdots k_{m}^{2m-1} \leq x \\ 
k_1^m k_2^{m+1}\cdots k_{m}^{2m-1}  \equiv r \md q  } }  \mu(k_2\cdots k_m)^2.$$
Fixing the values of $k_2,\ldots, k_m$, this becomes 
$$\sum_{\substack{ k_2,\ldots, k_m \in \N 
\\k_2^{m+1}\cdots k_{m}^{2m-1} \leq x  } } 
 \mu(k_2\cdots k_m)^2
\Osum_{\substack{ y\md q}} 
\#\left\{k_1\leq (xk_2^{-m-1}\cdots k_{m}^{-2m+1} )^{1/m}: k_1 \equiv y \md q \right\},
$$ where       $\Osum$ is subject to $ y^m k_2^{m+1}\cdots k_{m}^{2m-1}  \equiv r \md q$.
We obtain $$\sum_{\substack{ k_2,\ldots, k_m \in \N \\k_2^{m+1}\cdots k_{m}^{2m-1} \leq x  } }  \mu(k_2\cdots k_m)^2
\sum_{\substack{ y\md q}} \left( \frac{(xk_2^{-m-1}\cdots k_{m}^{-2m+1} )^{1/m}}{q} +O(1) \right)
.$$ For any $\frac{1}{m(m+1)}>\epsilon>0$ the error term contributes $$\ll \#\left\{ k_2,\ldots, k_m \in \N :(k_2\cdots k_{m})^{m+1} \leq x  \right\}
\leq \sum_{a\leq x^{1/(m+1)}}\tau_{m-1}(a)\ll_{\varepsilon} x^{\frac{1}{m+1}+\epsilon }=o(x^{1/m})
$$ where $\tau_{m-1}(a)$ is the number of ways of writing $a$ as the product of $m-1$ positive integers
and we used the standard bound $\tau_{m-1}(a)\ll_{\varepsilon} a^\epsilon$. Recall that we are applying Theorem~\ref{pergolesisalvereginainaminor} and hence only need to work with fixed modulus $q$ and don't need to record the dependence on $q$ in the error terms. The main term equals 
$$\frac{x^{1/m}}{q}
\sum_{\substack{ k_2\ldots, k_m \in \N 
\\
k_2^{m+1}\cdots k_{m}^{2m-1} \leq x  } } 
 \mu(k_2\cdots k_m)^2
 (k_2^{-m-1}\cdots k_{m}^{-2m+1} )^{1/m} \Osum_{\substack{ y\md q}} 1
 .$$  By the trivial bound $\displaystyle\Osum_{\substack{ y\md q}}1\leq q$ we see that the sum is 
 convergent as $x\to \infty$, hence we have obtained
$$ \#\{ k \in \N \cap[1,x]  : k \textrm{ is } m\textrm{-full }, k\equiv r \md q  \} = \rho(r,q)x^{1/m} +o(x^{1/m} ),$$
where $$\rho(r,q)=\frac{1}{q}\sum_{\substack{ k_2,\ldots, k_m \in \N   } } \frac{ \mu(k_2\cdots k_m)^2}{
 k_2^{1+1/m} k_3^{1+2/m} \cdots k_{m}^{ 2-1/m }  }\#\left\{ y\in \Z/q\Z : y^m k_2^{m+1}\cdots k_{m}^{2m-1} = r  \right\}.$$
Letting  $\omega(t) = \frac{1}{\displaystyle m t^{1/m-1}} $ we can then verify the remaining assumption of Theorem~\ref{pergolesisalvereginainaminor}
as in~\eqref{leonardoleosalveregina}.

 \subsection{The proof of Theorem~\ref{thm:nonalgebraic}}
 \label{ss:proofnonalg}
Let $k$ be the indicator function of the integers that are an integer power of $2$.
Then for all $q\in \N $ and $r\in [1,q]$ we have 
$$\sum_{\substack{ m\leq x\\m\equiv r \md q  }} k(m)
=\#\left\{0\leq t \leq \frac{\log x}{\log 2} :  2^t \equiv  r \md q  \right\}
.$$  Assume that $v_2(r)< v_2(q)$. Then the congruence is equivalent to 
$2^{t-v_2(r) } \equiv  r2^{-v_2(r)} \md {q 2^{-v_2(r)}}  $. If there exists $t>v_2(r)$ with such a property then 
since both $2^{t-v_2(r) }, q 2^{-v_2(r)}$ are even, one sees that $r2^{-v_2(r)}$ is even, which is a contradiction.
Hence, the sum is $$O(v_2(q) )+ 
\#\left\{v_2(q)\leq t \leq \frac{\log x}{\log 2} :  2^{t-v_2(q)} \equiv  r2^{-v_2(q)} \md {q2^{-v_2(q)}}  \right\}\mathds 1 (v_2(r) \geq v_2(q)) .$$
We can now define   $$\langle 2\rangle =\left\{2^j\md {q2^{-v_2(q)}}: j\in \N\right\}.$$ Then we must have   $r2^{-v_2(q)}\in \langle 2\rangle$, 
so we can write $r2^{-v_2(q)}\equiv 2^\alpha \md{q2^{-v_2(q)}}$ for some integer~$\alpha$.
Then the cardinality becomes  $$
O(v_2(q) )+ 
 \#\left\{v_2(q)\leq t \leq \frac{\log x}{\log 2} :  2^{t-v_2(q)-\alpha} \equiv 1 \md {q2^{-v_2(q)}}  \right\}\mathds 1 (v_2(r) \geq v_2(q),r2^{-v_2(q)}\in \langle 2\rangle).$$ 
Denoting the order of $2\md{q2^{-v_2(q)}}$ by $g(q)$ this then becomes 
$$ \frac{\mathds 1 (v_2(r) \geq v_2(q),r2^{-v_2(q)}\in \langle 2\rangle) }{g(q)} \frac{\log x}{\log 2} +O(v_2(q)).$$ Note that $v_2(q)\leq (\log q)/(\log 2)$. 
Therefore, if we let $$ \omega(t) =\frac{1}{t \log 2 }, \quad  \rho(r,q)= \frac{\mathds 1 (v_2(r) \geq v_2(q),r2^{-v_2(q)}\in \langle 2\rangle)}{g(q)} $$
we have shown that 
 $$ \sum_{\substack{ 1\leq m \leq x \\ m\equiv r \md q } } k(m)  = \rho(r,q) \int_1^{x} \omega(t) \mathrm d t +O(\log q ),$$
with an absolute implied constant. To use Theorem~\ref{pergolesisalvereginainaminor} we must verify the remaining assumption regarding $\int 
\omega$.  We will     show something more, namely,   \beq{eq:leonardoleomotet}{ \lim_{P\to \infty}
\frac{\c J(P)}{\log P}  =  d \sigma_\infty,} where $$ \c J(P):=P^d\int\limits_{\substack{ \b t \in \c B\\   F(\b t ) > P^{-d} }  } \omega(P^d f(\b t ) )\mathrm d \b t.$$
By~\eqref{darksun}  we have  
$$\c J(P)=\frac{1}{\log 2} 
\int_{ \gamma   \in \R } I(\c B,\gamma)
\int_{    P^{-d}    }^b  \frac{\mathrm e(-     \gamma    \mu     )}{\mu} \mathrm d   \mu \mathrm d \gamma
.$$ By the trivial bound $|\mathrm e(-   \gamma    \mu     ) |\leq  1 $
we deduce for  $\gamma\in \R$ and all $\lambda \in (0,b)$ that 
 $$ \int_{    \lambda }^b  \frac{ \mathrm e(-   \gamma    \mu     ) }{\mu} \mathrm d   \mu  \ll \log \frac{b}{\lambda } $$ with an absolute implied constant.
Hence, by~\eqref{bchrand4}, the contribution of $|\gamma|> \sqrt{\log P} $ towards~$\c J(P)$ is 
$$ \ll (\log P) \int_ {\gamma> \sqrt{\log P } } |I(\c B;\gamma)|   \mathrm d \gamma
=o(\log P).$$  
Similarly, the contribution of $ \mu> 1/\log P $ is 
$$\ll \int_{ |\gamma| \leq\sqrt{\log P }  } |I(\c B;\gamma)|
\int_{   1/\log P }^b  \frac{ \mathrm d   \mu }{\mu} 
\mathrm d \gamma \ll \log \log P .$$ Thus, 
  $$\c J(P)= \frac{1}{ \log 2 }
\int_{ |\gamma| \leq \sqrt{\log P }  } I(\c B;\gamma)
\int_{P^{-d} } ^ {  1/\log P  }   \frac{ \mathrm e(-   \gamma    \mu     ) }{\mu} \mathrm d   \mu 
\mathrm d \gamma+o(\log P).$$ In this range  we have $|\gamma \mu | \leq 1/\sqrt{\log P }$, hence, 
$$ \mathrm e(-   \gamma    \mu  )=1+O\left(1/\sqrt{\log P }\right) 
.$$ Substituting in the last expression for $\mathcal J$ leads us to  $$ \c J(P)=
\frac{1}{ \log 2 } \int_{ |\gamma| \leq \sqrt{\log P }  } I(\c B;\gamma)
\int_{P^{-d} } ^ {  1/\log P  }  \frac{   \mathrm d   \mu  }{\mu} 
\mathrm d \gamma+ O(E),$$ where $E$ is given by  $$\frac{1}{\sqrt{\log P } }
 \int_{ |\gamma| \leq \sqrt{\log P }  }| I(\c B;\gamma)| \int_{P^{-d} } ^ {  1 }  \frac{  \mathrm d   \mu    }{\mu} 
\mathrm d \gamma   
\ll \frac{1}{\sqrt{\log P } }
\int_{P^{-d}}^1  \frac{  \mathrm d   \mu    }{\mu}
\ll\sqrt{\log P }
 =o(\log P).$$The   main  term is 
\begin{align*}
&\frac{1}{ \log 2 }
\int_{ |\gamma| \leq \sqrt{\log P }  } I(\c B;\gamma)
(-\log \log P +d \log P) 
\mathrm d \gamma
\\
=& O(\log \log P ) 
+\frac{d \log P }{\log 2 }
\int_{ |\gamma| \leq \sqrt{\log P }  } I(\c B;\gamma)
 \mathrm d \gamma=
 \frac{d \sigma_\infty(f) }{ \log 2 }
 \log P+o(\log P)
,\end{align*} which concludes the proof of~\eqref{darksun}.

 We are   now in position to apply Theorem~\ref{pergolesisalvereginainaminor}.
Before doing so, we  simplify $\rho(r,q)$ by noting that 
$$ \frac{1}{\phi\left(q 2^{-v_2(q) }\right) } \sum_{1\leq j \leq \phi\left(q 2^{-v_2(q) }\right) } \mathds 1 \left (r 2^{-v_2(q)} \equiv 2^j \md{q 2^{-v_2(q) } } \right)
=\frac{\mathds 1(r 2^{-v_2(q)} \in \langle 2 \rangle)}{g(q  )},$$
where $\phi $ is Euler's totient function. Hence, we may write 
$$\rho(r,q)=
\frac{\mathds 1(v_2(q) \leq v_2(r) ) }{\phi\left(q 2^{-v_2(q) }\right) } 
\#\left \{1\leq j \leq \phi(q 2^{-v_2(q) }) : 
r 2^{-v_2(q)} \equiv 2^j \md{q 2^{-v_2(q) } }  
\right \}.$$  Let $t_p(z)$ and $T_z$ be as in  Theorem~\ref{pergolesisalvereginainaminor}
and let  $w=T_z 2^{-t_2(z)}$.
Each $\b t\in (\Z/T_z\Z)^n$ can be uniquely written as $\b t = w \b x + 2^{t_2(z)} \b y $ where 
$\b x \in (\Z/2^{t_2(z)} \Z)^n$ and $\b y \in (\Z/w\Z)^n$. We find that  
 $$ \sum_{\b t \in (\Z/T_z\Z )^{n} } \frac{ \rho( f(\b t ) ,T_z) }{ T_z^{n-1} }
=
\frac{\#\left\{\b x \in (\Z/2^{t_2(z)}\Z )^{n} : f(\b x ) = 0\right\}}{2^{(n-1)t_2(z)}}
\frac{\#\left\{\b y \in (\Z/w\Z )^{n}, h   \in \Z/\phi(w)\Z  :  f(\b y )   =  2^ h \right\}}{\phi(w) w^{n-1}}
.$$The first fraction in the right hand side converges to $\sigma_2(f)$ as $z\to \infty$. By the Chinese remainder theorem the second fraction is  
$$\prod_{2<p\leq z } \frac{\#\left\{\b y \in (\Z/p^{t_p(z)}\Z )^{n}, h   \in \Z/\phi(p^{t_p(z)})\Z  :  f(\b y )   =  2^ h \right\}}
{(p-1) p^{n t_p(z) }}
$$  and letting $z\to \infty$ completes the proof.

\end{document}